\documentclass{article}
\usepackage[utf8]{inputenc}
\usepackage[a4paper,left=2.0cm,right=2.0cm,top=2.8cm,bottom=2.8cm]{geometry}
\usepackage{graphicx,epstopdf}
\usepackage{amsmath,amssymb,amsthm}
\usepackage{dsfont}
\usepackage{cite}
\usepackage{hyperref}
\usepackage{longtable}
\usepackage{subfig}
\usepackage{geometry}
\usepackage{caption}
\numberwithin{figure}{section}
\numberwithin{table}{section}

\usepackage[T1]{fontenc}
\usepackage{authblk}

\date{} 

\def\0{{\mathbf{0}}}

\def\e{{\mathbf{e}}}
\def\u{{\mathbf{u}}}

\def\n{{\mathbf{n}}}
\def\vv{{\mathbf{v}}}

\allowdisplaybreaks[1]
\numberwithin{equation}{section}
\newtheorem{theorem}{Theorem}[section]

\newtheorem{lemma}{Lemma}[section]

\newtheorem{re}{Remark}[section]

\title{A high-order, high-efficiency adaptive time filter algorithm for shale reservoir model based on coupled fluid flow with porous media flow
\thanks{Supported in part by NSF of China (No. 11771259 and No. 12001347), Shaanxi Provincial Joint Laboratory of Artificial Intelligence (No. 2022JC-SYS-05), Innovative team project of Shaanxi Provincial Department of Education (No. 21JP013 and No. 21JP019), Shaanxi Province Natural Science basic research program key project (No. 2023-JC-ZD-02), National High-end Foreign Experts Recruitment Plan (No. G2023041032L), Energy Mathematics and Data Fusion Key Laboratory of Higher Education in Shaanxi Province, Shaanxi Provincial Demonstration Base for the Introduction of Foreign Intelligence: Mathematics and data science cross-integration innovation and intelligence introduction base.}
}
\author{Jian Li\thanks{School of Mathematics and Data Science, Shaanxi University of science and technology, Xi'an, Shaanxi 710021, China.({\tt jianli@sust.edu.cn})}
, Lele Chen\thanks{School of Mathematics and Data Science, Shaanxi University of science and technology, Xi'an, Shaanxi 710021, China.({\tt 221711036@sust.edu.cn})}
, Yi Qin\thanks{School of Mathematics and Data Science, Shaanxi University of science and technology, Xi'an, Shaanxi 710021, China.({\tt yiqin@sust.edu.cn})}
\;\;and Zhangxin Chen\thanks{Department of Chemical and Petroleum Engineering, Schulich School of Engineering, University of Calgary, Calgary T2N 1N4, Canada.({\tt zhachen@ucalgary.ca})}
}

\begin{document}
\maketitle
\begin{abstract}
In this paper, a third-order time adaptive algorithm with less computation, low complexity is provided for shale reservoir model based on coupled fluid flow with porous media flow. The algorithm combines the three-step linear time filters method for simple post-processing and the second-order backward differential formula (BDF2), is third-order accurate and provides, at no extra computational complexity. At the same time, the time filter method can also be used to damp non-physical oscillations inherent in the BDF2 method, ensuring stability. We proves the variable time stepsize second-order backward differential formula plus time filter (BDF2-TF) algorithm's stability and the convergence properties of the fluid velocity $\u$ and hydraulic head $\phi$ in the $L^2$ norm with an order of $O(k_{n+1}^3 + h^3)$. In the experiments, the adaptive algorithm automatically adjusts the time step in response to the varying characteristics of different models, ensuring that errors are maintained within acceptable limits. This algorithm addresses the issue that high-order algorithms may select inappropriate time steps, resulting in instability or reduced precision of the numerical solution, thereby enhancing calculation accuracy and efficiency. We perform three-dimensional numerical experiments to verify the BDF2-TF algorithm's effectiveness, stability, and third-order convergence. Simultaneously, a simplified model is employed to simulate the process of shale oil extraction from reservoirs, further demonstrating the algorithm's practical applicability.
\end{abstract}



\section{Introduction}

Shale oil, an unconventional hydrocarbon resource, is extracted from shale rock formations \cite{1,32,2}. The low permeability of these formations presents significant challenges for conventional extraction methods. To address these challenges, hydraulic fracturing and horizontal well fracturing are extensively employed to increase the permeability and production. Optimizing shale oil production requires an accurate mathematical model to simulate fluid flow within the shale matrix. This flow involves fluid flow through wellbores and fractures during hydraulic fracturing, as well as porous media flow characterizing hydrocarbon migration and diffusion through natural fractures and matrix pores. To precisely simulate these processes, it is essential to develop coupled models combining both fluid flow mechanisms and porous media flow phenomena. The Stokes-Darcy model \cite{5,33,3,6} provides a fundamental framework for describing such systems. The free fluid flow region is governed by the evolutionary Stokes equation \cite{28,29,27,30}, while the porous media region is described by the Darcy equation. To couple the porous medium with conduit region, three physically valid coupling conditions, are utilized to capture the interfacial phenomena efficiently and provide a comprehensive description of the fluid dynamics within shale reservoirs. In this paper, we investigate the coupled Stokes-Darcy model within a bounded region $\Omega=\Omega_{p}\bigcup\Omega_{f}\subset\mathbb{R}^d$, where $d=2,3$. Here, $\Omega_p$ represents the porous media region, while $\Omega_f$ denotes the fluid region. The dynamics of free flow in $\Omega_f$ are governed by the unsteady Stokes equations: find the fluid velocity $\u$ and kinematic pressure $p$ that satisfy
\begin{align}\label{Stokes}
	&\frac{\partial\u}{\partial t}-\nu\Delta\u+\nabla p =\mathbf{f}_1, \quad\nabla\cdot\u =0,\quad \text{in}\;\Omega_f\times [0,T],
\end{align}
where $\nu>0$ denotes the kinetic viscosity, and $\mathbf{f}_1$ signifies the external force. The flow of porous media in $\Omega_p$ is described by the primary Darcy equations: find the hydraulic head $\phi$ such that
\begin{align}\label{Darcylaw}
	&S\frac{\partial\phi}{\partial t}-\nabla\cdot(\textbf{K}\nabla\phi)=f_2,\quad \text{in}\;\Omega_f\times [0,T].
\end{align}
Here, $S$ represents the specific mass storativity coefficient, $\textbf{K}$ denotes the hydraulic conductivity tensor, and $f_2$ is a source term.

On the interface $\Gamma=\partial\Omega_f\cap\partial\Omega_p$, the three following interface conditions are enforced \cite{14,12,33}:
\begin{equation}\label{BJS}
	\left\{\begin{array}{c}
		\u\cdot\n_f-(\textbf{K}\nabla\phi)\cdot\n_p=0,\\
		p-\nu\n_f\cdot\frac{\partial\u}{\partial\n_f}=g\phi,\\
		-\nu\tau_j\cdot\frac{\partial\u}{\partial\n_f}=\frac{\alpha\nu\sqrt{d}}{\sqrt{trace(\textbf{K})}}\u\cdot\tau_j.
	\end{array}\right.
\end{equation}
The unit outer normal vectors to the porous media and fluid regions at the interface $\Gamma$ are denoted by $\n_f$ and $\n_p$, respectively. The mutually orthogonal unit tangential vectors to the interface $\Gamma$ are represented by $\tau_j (j=1,\ldots,d-1)$. A positive parameter dependent on medium properties is denoted by $\alpha$, which must be experimentally determined. $g$ represents gravitational acceleration.

Numerous numerical algorithms have been developed for both the stationary and non-stationary Stokes-Darcy model \cite{7,8,10,9,15,11,16}. Regarding time discretization, extensive research has focused on algorithms with constant time steps of either first or second order. For instance, Mu and Zhu \cite{3} consucted the first-order decoupled algorithm using the Backward Euler method. Qin et al. \cite{17} studied the second-order coupled algorithm of backward Euler method plus time filter. Li and Hou proposed a second-order backward differential formula (BDF2) for evoluting the Stokes-Darcy system \cite{18}. However, constant time step algorithm often require the selection of smaller time steps to achieve optimal convergence order and minimize error. As a result, many researchers have investigated variable time step algorithms. Qin \cite{19} initially proposed a second-order time filter algorithm with variable time steps for closed-loop geothermal systems. A second-order variable time step BDF2 numerical scheme for the Cahn-Hilliard equation was provided and analyzed in \cite{21}. An adaptive algorithm, a special form of variable time step algorithm, automatically adjusts the time step based on model characteristics during calculations. This algorithm ensures error control within acceptable limits, minimizes step sizes, and improves computational efficiency \cite{23,6,22}. Presently, low-order algorithms have been thoroughly studied, research of higher order algorithms remains limited. High-order algorithms can significantly enhance the precision and efficiency of numerical simulation, thereby giving the analysis and prediction of complex systems more dependable and precise. The backward differential formula (BDF) have been extensively utilized in numerical methods for time-dependent partial differential equations (PDEs) \cite{25,26,24}. These formulas offer higher order accuracy without significantly increasing computational complexity. However, the limitations of high-order algorithms include numerical instability in the presence of rapid changes or oscillations, as well as the potential accumulation of errors leading to deviation from the true solution during long-term simulations. The time filter method can effectively solve phenomenons and improve temporal accuracy and computational efficiency with minimal additional implementation effort.

By integrating the strengths of high-order algorithms, time filter method, and adaptive algorithm, we have developed and analyzed a third-order adaptive second-order backward differential formula plus time filter (BDF2-TF) algorithm for coupled incompressible flow with porous media flow. The key advantages of our algorithm are as follows:

$\bullet$ The selection of time steps is crucial for high-order algorithms, as inappropriate choices can result in instability or reduced precision of numerical solutions. In contrast, time-adaptive algorithms automatically adjust the step size based on the dynamic characteristics of solutions, thereby enhancing calculation accuracy and efficiency while ensuring stability.

$\bullet$ Although the high-order algorithm boasts high theoretical accuracy, it exhibits certain deficiencies in terms of numerical stability, error accumulation, and oscillation phenomena. The time filter method effectively addresses these issues, and enhances time precision and computational efficiency without increasing the additional complexity. Consequently, high-order algorithms demonstrate enhanced robustness and efficiency in practical applications.

$\bullet$ In theory, stability is ensured when the time step ratio satisfies $0<\tau_n\leq 1.0315$, and a convergence property of fluid velocity $\u$ and hydraulic head $\phi$ in the $L^2$-norm with an order of $O(k_{n+1}^3+h^3)$ has been deduced.

$\bullet$ In the experiment, we use 3D examples for the first time to study the effectiveness, stability and convergence order of the third-order BDF2-TF algorithm, obtaining results consistent with the theoretical analysis. Additionally, a simplified model is employed to simulate the process of extracting shale oil from reservoirs.

The paper is organized as follows: a review of the weak form of the Stokes-Darcy model is presented in Section \ref{sec1:model}. Section \ref{sec2:algorithm} proposes the second-order backward differential formula plus time filter (BDF2-TF) and an adaptive algorithm for the Stokes-Darcy model. Section \ref{sec3:anal} proves the the stability and error estimation of the variable time step BDF2-TF algorithm. Section \ref{sec5:experiments} presents numerical experiments to demonstrate the effectiveness, stability, and convergence of proposed algorithm.

\section{The weak formulation}
\label{sec1:model}

To obtain the weak formulation for the model problems (\ref{Stokes}) - (\ref{BJS}), we introduce function spaces over the domain $D$ as follows:
\begin{align*}
	&\mathbf{H}_{f}=\left\{\vv \in \left(\mathbf{H}^{1}\left(\Omega_{f}\right)^{d}\right): \vv=0 \mbox { on } \partial \Omega_{f}\backslash\Gamma\right\}, \\
	&H_{p}=\left\{\varphi \in H^{1}\left(\Omega_{p}\right)^{d}: \varphi=0 \mbox { on } \partial \Omega_{p}\backslash\Gamma\right\}, \\
	&Q=L_0^{2}\left(\Omega_{f}\right),\qquad \mathbf{W}=\mathbf{H}_f\times H_p.
\end{align*}
For the domain $D$, $(\cdot,\cdot)_{D}$ refers to the scalar inner product in $D$ for $D=\Omega_f$ or $\Omega_p$.  In particular, we denote the $H^1(\Omega_{f/p})$ norm by $\|\cdot\|_{\mathbf{H}_f/H_p}$, the $L^2(\Gamma)$ norm by $\|\cdot\|_{\Gamma}$ and the $L^2(\Omega_{f/p})$ norm by $\|\cdot\|_{f/p}$, and define the corresponding norms and the notation hereafter:
\begin{align}
	&\|\u\|_{\mathbf{H}_f}=\|\nabla\u\|_{L^2\left(\Omega_f\right)},\qquad \|\u\|_f=\|\u\|_{L^2\left(\Omega_f\right)},\\
	&\|\phi\|_{H_p}=\|\nabla\phi\|_{L^2\left(\Omega_p\right)}, \qquad \|\phi\|_p=\|\phi\|_{L^2\left(\Omega_p\right)}.
\end{align}

The weak formulation of the coupled Stokes-Darcy problem is given as follows: find $w=\left(\u,\phi\right)\in \mathbf{W}$ and $p\in Q$ such that $\forall\;t\in (0,T]$,
\begin{align}\label{wf}
	&(w_t,z)+a(w,z)+b(z,p)=(f,z),\quad \forall\;z=(\vv,\varphi)\in\mathbf{W},\\
	&b(w,q)=0, \quad \forall\;q\in Q,
\end{align}
where
\begin{align*}
	(w_t,z)&=(\frac{\partial\u}{\partial t},\vv)_{\Omega_f}+gS(\frac{\partial\phi}{\partial t},\varphi)_{\Omega_p},\\
	a(w,z)&=a_f(\u,\vv)+a_p(\phi,\varphi)+c_{\Gamma}(w,z),\\
	a_f(\u,\vv)&=\nu(\nabla\u,\nabla\vv)_{\Omega_f}+\sum\limits_{j=1}^{d-1}\int_{\Gamma}\frac{\alpha\sqrt{\nu g}}{\sqrt{trace(\textbf{K})}}(\u\cdot\tau_j)(\vv\cdot\tau_j),\\
	a_p(\phi,\varphi)&=g(\textbf{K}\nabla\phi,\nabla\varphi)_{\Omega_p},\\
	c_{\Gamma}(w,z)&=c_{\Gamma}\left(\u,\phi;\vv,\varphi\right)= g\int_{\Gamma}\left(\phi\vv\cdot\n_f-\varphi\u\cdot\n_f\right),\\
	b(z,p)&=-(p,\nabla\cdot\vv)_{\Omega_f},\\
	(f,z)&=(\mathbf{f}_1,\vv)_{\Omega_f}+g(f_2,\varphi)_{\Omega_p}.
\end{align*}

We recall the Poincar$\acute{e}$, trace, and Sobolev inequalities that are useful in the following analysis. There exist constants $C_p$, $C_t$, and $C_s$, which depend only on the domain $\Omega_f$, and $\tilde{C}_p$, $\tilde{C}_t$, and $\tilde{C}_s$, which depend
only on the domain $\Omega_p$,  such that, for all $\vv\in \mathbf{H}_f$ and $\varphi\in H_p$,
\begin{align}\label{in1}
	&\left\|\vv\right\|_f\leq C_p\left\|\vv\right\|_{\mathbf{H}_f},\qquad\qquad\quad\quad\;\; \left\|\varphi\right\|_f\leq \tilde{C}_p\left\|\varphi\right\|_{H_p},\\
	&\left\|\vv\right\|_{\Gamma}\leq C_t\left\|\vv\right\|_{f}^{\frac{1}{2}}\left\|\vv\right\|_{\mathbf{H}_f}^{\frac{1}{2}},\qquad \qquad\; \left\|\varphi\right\|_{\Gamma}\leq \widetilde{C}_t\left\|\varphi\right\|_{p}^{\frac{1}{2}}\left\|\varphi\right\|_{H_p}^{\frac{1}{2}},\\
	&\left\|\vv\right\|_{\Gamma}\leq C_s\left\|\vv\right\|_{f}^{\frac{1}{2}}\left\|\nabla\vv\right\|_{f}^{\frac{1}{2}},\qquad\qquad \left\|\varphi\right\|_{\Gamma}\leq \tilde{C}_s\left\|\varphi\right\|_{p}^{\frac{1}{2}}\left\|\nabla\varphi\right\|_{p}^{\frac{1}{2}}.
\end{align}

$\textbf{K}$ is uniformly bounded and positive defined in $\Omega_p$: there exist $k_{min},k_{max}>0$ such that
\begin{equation}\label{k}
	k_{\min}|x|^{2}\leq \textbf{K} x\cdot x\leq k_{\max}|x|^{2}, \qquad  x\in\Omega_p.
\end{equation}
From (\ref{k}), we have
\begin{equation}\label{kmin}
	\frac{1}{\sqrt{k_{\max}}}\left\|\textbf{K}^{\frac{1}{2}}\nabla\phi\right\|_{L^2(\Omega_p)} \leq\left\|\phi\right\|_{H_p}\leq\frac{1}{\sqrt{k_{\min}}}\left\|\textbf{K}^{\frac{1}{2}}\nabla\phi\right\|_{L^2(\Omega_p)}, \quad \forall\;\phi\in H_p.
\end{equation}

\section{The time filter numerical algorithm}
\label{sec2:algorithm}

To discretize the Stokes-Darcy problem in space by finite element method, let $h_i$ be a positive parameter and ${\cal T}_{h_i}$ be quasi-uniform partition of triangular or quadrilateral elements of $\Omega_i$, $i=f,p$. We assume ${\cal T}_{h_f}$ and ${\cal T}_{h_p}$ match at the interface and the interface is straight in the analysis. For the fluid velocity $\u$ and hydraulic head $\phi$, we utilize the continuous piecewise cubic finite element (P3) spaces $\mathbf{H}_{fh}\subset \mathbf{H}_f$ and $H_{ph}\subset H_p$, while continuous piecewise quadratic functions (P2) in space $Q_{h}\subset Q$ represents the pressure $p$. It should be noted that the space pair $(\mathbf{H}_{fh},Q_h)$ satisfies the discrete inf-sup condition: there exist constants $\beta$ independent of mesh size $h$, such that
$\inf\limits_{q_h\in Q_h}\sup\limits_{\vv_h\in \mathbf{H}_{fh}}\frac{(q_h,\nabla\cdot\vv_h)_{\Omega_f}}{\left\|q_h\right\|_{Q}\left\|\vv_h\right\|_{\mathbf{H}_f}}\geq\beta,$ and define $\mathbf{W}_h=\mathbf{H}_{fh}\times H_{ph}.$

Furthermore, we will use the inverse inequalities: there exist constants $C_I$ and $\tilde{C}_I$, which are dependent on the angles in the finite element mesh, such that, for all $\vv\in \mathbf{H}_{fh}$ and $\phi\in H_{ph}$,
\begin{align}\label{inv}
	&\|\vv_h\|_{\mathbf{H}_f}\leq C_Ih^{-1}\|\vv_h\|_f,\qquad \|\phi_h\|_{H_p}\leq \tilde{C}_Ih^{-1}\|\phi\|_p.
\end{align}

Let $P=\{t_n\}_{n=0}^N$ represent the time partition on the interval $[0,T]$, where $t_0=0$, $t_N=T$, and $k_{n+1}=t_{n+1}-t_n$ denotes the time step size. The ratio of time steps is denoted as $\tau_n=\frac{k_{n+1}}{k_{n}}$. Here, $\u_h^{n+1}$, $p_h^{n+1}$, and $\phi_h^{n+1}$ denotes the discrete approximation by the scheme to $\u_h(t_{n+1})$, $p_h(t_{n+1})$, and $\phi_h(t_{n+1})$. The details are presented below for the variable time step second-order backward differentiation formula plus time filter (BDF2-TF) algorithm and the adaptive BDF2-TF algorithm.
~\\
~\\
\textbf{Algorithm 1: A second-order backward differential formula plus time filter algorithm (BDF2-TF algorithm)}\label{BDF2-TF}

$\bullet$ Given $(\u_{h}^{0}, p_{h}^{0}, \phi_{h}^0)$, $(\u_{h}^{1}, p_{h}^{1}, \phi_{h}^1)$, and $(\u_{h}^{2}, p_{h}^{2}, \phi_{h}^2)$, find $(\hat{\u}_{h}^{n+1}, \hat{p}_{h}^{n+1},\hat{\phi}_{h}^{n+1})\in (\mathbf{H}_{fh},Q_{h},H_{ph})$, with $n=2,\ldots,N-1$, such that $\forall\;(\vv_{h},q_{h}, \varphi_h)\in (\mathbf{H}_{fh},Q_{h},H_{ph})$,
\begin{align}\label{s1}
	&\frac{1}{k_{n+1}}\left(\frac{1+2\tau_n}{1+\tau_n}\hat{\u}_{h}^{n+1}-\left(1+\tau_n\right)\u_h^n +\frac{\tau_{n}^2}{1+\tau_n}\u_{h}^{n-1},\vv_h\right)+a_f\left(\hat{\u}_{h}^{n+1},\vv_h\right) +b\left(\vv_h,\hat{p}_{h}^{n+1}\right)\nonumber\\
	=&\left(\mathbf{f}_1^{n+1},\vv_h\right)-g\int_{\Gamma}\phi_{h,\sigma}^n\vv_h\cdot\n_f,\\
	&b(\hat{\u}_h^{n+1},q_h)=0, \label{s11}\\
	&\frac{gS}{k_{n+1}}\left(\frac{1+2\tau_n}{1+\tau_n}\hat{\phi}_{h}^{n+1}-\left(1+\tau_n\right)\phi_h^n +\frac{\tau_{n}^2}{1+\tau_n}\phi_{h}^{n-1},\varphi_h\right)+a_p\left(\hat{\phi}_{h}^{n+1},\varphi_h\right) \nonumber\\
	=&g\left(f_2^{n+1},\varphi_h\right)+g\int_{\Gamma}\varphi_h\u_{h,\sigma}^n\cdot\n_f.\label{s2}
\end{align}
Where $w_{h,\sigma}^n$ $(w=\u\;\text{or}\;\phi)$ is a third-order approximate time format (below $\delta w_h^{n}:=w_h^n-w_h^{n-1}$)
\begin{align}\label{jjmxs}
	w_{h,\sigma}^n=&\frac{\left(1+\tau_n\right)\left(1+\tau_{n-1}\left(1+\tau_n\right)\right)}{1+\tau_{n-1}} w_{h}^n-\tau_n\left(1+\tau_{n-1}\left(1+\tau_n\right)\right)w_{h}^{n-1} +\frac{\tau_n\tau_{n-1}^2\left(1+\tau_n\right)}{1+\tau_{n-1}}w_{h}^{n-2}\nonumber\\
	=&w_{h}^n+\frac{\tau_n\left(1+\tau_{n-1}\left(2+\tau_n\right)\right)} {1+\tau_{n-1}}\left(w_{h}^n-w_{h}^{n-1}\right) -\frac{\tau_n\tau_{n-1}^2\left(1+\tau_n\right)}{1+\tau_{n-1}}\left(w_{h}^{n-1}-w_{h}^{n-2}\right)\nonumber\\
	=&\sigma_3w_{h}^n+\sigma_2\delta w_{h}^n-\sigma_1\delta w_{h}^{n-1}.
\end{align}

$\bullet$ Apply time filter to update the previous solutions
\begin{align}\label{s3}
	&\Phi_{h}^{n+1}=\hat{\Phi}_h^{n+1}-\eta^3\varrho^3\hat{\Phi}_h^{n+1}, \qquad \Phi=\u,\;p,\;\text{or}\;\phi.
\end{align}

By employing algebraic manipulation, the aforementioned equality can be expressed in relation to the stepsize ratio $\tau$ as follows:
\begin{align}\label{s4}
	&\Phi_{h}^{n+1}=\hat{\Phi}_h^{n+1}+\alpha\left\{\frac{6}{\left(\tau_{n}+1\right)\left(1+\tau_{n-1}\left(\tau_n+1\right)\right)}\hat{\Phi}_h^{n+1} -\frac{6}{1+\tau_{n-1}}\Phi_h^{n}\right.\nonumber\\
	&\qquad\qquad\qquad\qquad\left.+\frac{6\tau_{n}}{1+\tau_n}\Phi_h^{n-1}-\frac{6\tau_{n-1}^2\tau_n}{\left(1+\tau_{n-1}\right)\left(1+\tau_{n-1}\left(\tau_n+1\right)\right)}\Phi_h^{n-2}\right\},
\end{align}
where $\alpha=-\frac{1}{6}\frac{\tau_n\tau_{n-1}\left(1+\tau_n\right)^2\left(1+\tau_{n-1}\left(1+\tau_n\right)\right)}{\left(1+2\tau_n\right)\left(1+\tau_{n-1}\left(1+\tau_n\right)\right)+\tau_n\tau_{n-1}\left(1+\tau_n\right)}$. It follows that $\hat{\Phi}_h^{n+1}$ can be expressed as an equivalent form
\begin{align}\label{hat}
	\hat{\Phi}_{h}^{n+1}=&\left(1+\frac{\tau_{n}\tau_{n-1}\left(1+\tau_n\right)}{\left(1+2\tau_n\right)\left(1+\tau_{n-1}\left(1+\tau_n\right)\right)}\right)\Phi_h^{n+1} -\frac{\tau_n\tau_{n-1}\left(1+\tau_n\right)^2}{\left(1+2\tau_n\right)\left(1+\tau_{n-1}\right)}\Phi_h^n\nonumber\\
	&+\frac{\tau_n^2\tau_{n-1}\left(1+\tau_n\right)}{1+2\tau_n}\Phi_h^{n-1} -\frac{\tau_n^2\tau_{n-1}^3\left(1+\tau_n\right)^2}{\left(1+2\tau_n\right)\left(1+\tau_{n-1}\right)\left(1+\tau_{n-1}\left(1+\tau_n\right)\right)}\Phi_h^{n-2}\nonumber\\
	=&\Phi_h^{n+1}+\frac{\tau_{n}\tau_{n-1}\left(1+\tau_n\right)}{\left(1+2\tau_n\right)\left(1+\tau_{n-1}\left(1+\tau_n\right)\right)}\delta\Phi_{h}^{n+1} -\frac{\tau_n^2\tau_{n-1}\left(1+\tau_n\right)\left(1+\tau_{n-1}\left(2+\tau_n\right)\right)} {\left(1+2\tau_n\right)\left(1+\tau_{n-1}\right)\left(1+\tau_{n-1}\left(1+\tau_n\right)\right)}\delta\Phi_{h}^{n}\nonumber\\
	&+\frac{\tau_n^2\tau_{n-1}^3\left(1+\tau_n\right)^2}{\left(1+2\tau_n\right)\left(1+\tau_{n-1}\right)\left(1+\tau_{n-1}\left(1+\tau_n\right)\right)}\delta\Phi_{h}^{n-1}\nonumber\\
	=&\Phi_h^{n+1}+\gamma_3\delta\Phi_{h}^{n+1} -\gamma_2\delta\Phi_{h}^{n} +\gamma_1\delta\Phi_{h}^{n-1}.
\end{align}
The first term on the left of equations (\ref{s1}) and (\ref{s2}) can be replaced with $\hat{\Phi}_h^{n+1}$, resulting in a third-order temporal accuracy scheme (below $\delta\Phi_h^{n+1}:=\Phi_h^{n+1}-\Phi_h^{n}$)
\begin{align}\label{sjjd}
	&\frac{1+2\tau_n}{1+\tau_n}\hat{\Phi}_{h}^{n+1}-\left(1+\tau_n\right)\Phi_h^n +\frac{\tau_{n}^2}{1+\tau_n}\Phi_{h}^{n-1}=\mathcal{A}(\Phi_h^{n+1})\nonumber\\
	=&\left(1+\frac{\tau_n}{1+\tau_n}+\frac{\tau_n\tau_{n-1}}{1+\tau_{n-1}\left(1+\tau_n\right)}\right)\Phi_h^{n+1} -\left(1+\tau_n+\frac{\tau_n\tau_{n-1}\left(1+\tau_n\right)}{1+\tau_{n-1}}\right)\Phi_h^{n}\nonumber\\
	&+\left(\tau_n^2\tau_{n-1}+\frac{\tau_n^2}{1+\tau_n}\right)\Phi_h^{n-1} -\frac{\tau_n^2\tau_{n-1}^3\left(1+\tau_n\right)} {\left(1+\tau_{n-1}\right)\left(1+\tau_{n-1}\left(1+\tau_n\right)\right)}\Phi_h^{n-2}\nonumber\\
	=&\left(\frac{1+2\tau_n}{1+\tau_n}+\frac{\tau_n\tau_{n-1}}{1+\tau_{n-1}\left(1+\tau_n\right)}\right)\delta\Phi_{h}^{n+1} -\left(\frac{\tau_n^2}{1+\tau_n}+\frac{\tau_n^2\tau_{n-1}\left(1+\tau_{n-1}\left(2+\tau_n\right)\right)} {\left(1+\tau_{n-1}\right)\left(1+\tau_{n-1}\left(1+\tau_n\right)\right)}\right)\delta\Phi_{h}^{n}\nonumber\\ &+\frac{\tau_n^2\tau_{n-1}^3\left(1+\tau_n\right)} {\left(1+\tau_{n-1}\right)\left(1+\tau_{n-1}\left(1+\tau_n\right)\right)}\delta\Phi_{h}^{n-1}\nonumber\\
	=&\beta_3\delta\Phi_{h}^{n+1}-\beta_2\delta\Phi_{h}^{n}+\beta_1\delta\Phi_{h}^{n-1}.
\end{align}
~\\
\textbf{Algorithm 2: Adaptive BDF2-TF Algorithm}\label{adaptive BDF2-TF}

Using Newton interpolation, the variable stepsize BDF methods of order $p\;(BDF-p)$ can be expressed as outlined in \cite{35}. The $j$th order divided difference denotes by $\varrho^j \u=\u\left[t_{n+m}, \ldots, t_{n+m-j}\right]$ and $\eta^{p+1}=\frac{\prod\limits_{i=1}^p\left(t_{n+m}-t_{n+m-i}\right)}{\sum\limits_{j=1}^{p+1}\left(t_{n+m}-t_{n+m-j}\right)^{-1}}$ is the parameter in adaptive algorithm. Define $\hat{\gamma}$ and $\check{\gamma}$ as two safety factors. $\check{\gamma}$ is used to prevent the next step size becoming too big to decrease the chance that the next solution will be rejected. On the other hand, $\hat{\gamma}$ is making the stepsize growing more slowly so that the recomputed solution is more likely to be accepted. Let $\epsilon$ represent the tolerance.
\begin{longtable}{l}
	\hline
	$\mathbf{Data}:$ $\epsilon$, $\hat{\gamma}=1.0$, $\check{\gamma}=0.5$, $\{\u_{h}^0,\u_h^1,\u_h^2\}$, $\{p_{h}^0,p_h^1,p_h^2\}$, $\{\phi_{h}^0,\phi_h^1,\phi_h^2\}$\\
	$\mathbf{Result:}$ $\u_{h}^{n+1}$, $p_h^{n+1}$, and $\phi_h^{n+1}$\\
	$\mathbf{Initialization:}$ $k_{n+1}\leftarrow t_{n+1}-t_n$, $k_n\leftarrow t_{n}-t_{n-1}$, $\varrho^j \u$ and $\eta^{p+1}$ from (\ref{s3})\\
	\qquad Compute $\u_h^{n+1}$, $p_h^{n+1}$, and $\phi_h^{n+1}$ by solving\\
	\\ $\qquad\qquad\frac{1}{k_{n+1}}\left(\frac{1+2\tau_n}{1+\tau_n}\hat{\u}_{h}^{n+1}-\left(1+\tau_n\right)\u_h^n +\frac{\tau_{n}^2}{1+\tau_n}\u_{h}^{n-1},\vv_h\right)+a_f\left(\hat{\u}_{h}^{n+1},\vv_h\right) +b\left(\vv_h,\hat{p}_{h}^{n+1}\right),$ \\
	$\quad\qquad=\left(\mathbf{f}_1^{n+1},\vv_h\right)-g\int_{\Gamma}\phi_{h,\sigma}^n\vv_h\cdot\n_f,$\\
	\qquad\qquad$ b\left(\hat{\u}_h^{n+1}, q_h\right)=0 $. \\
	\qquad\qquad$  \frac{gS}{k_{n+1}}\left(\frac{1+2\tau_n}{1+\tau_n}\hat{\phi}_{h}^{n+1}-\left(1+\tau_n\right)\phi_h^n +\frac{\tau_{n}^2}{1+\tau_n}\phi_{h}^{n-1},\varphi_h\right)+a_p\left(\hat{\phi}_{h}^{n+1},\varphi_h\right)$\\
	\quad\qquad$=g\left(f_2^{n+1},\varphi_h\right)+g\int_{\Gamma}\varphi_h\u_{h,\sigma}^n\cdot\n_f.$\\
	\\
	\qquad\qquad Time filter for $\hat{\u}_h^{n+1}$, $\hat{p}_h^{n+1}$, and $\hat{\phi}_h^{n+1}$\\
	\qquad\qquad\qquad\qquad\qquad\qquad\qquad$ \u_h^{n+1}\leftarrow\hat{\u}_h^{n+1}-\eta^3 \varrho^3 \hat{\u}_h^{n+1},$ \\
	\qquad\qquad\qquad\qquad\qquad\qquad\qquad$p_h^{n+1}\leftarrow\hat{p}_h^{n+1}-\eta^3 \varrho^3 \hat{p}_h^{n+1},$\\
	\qquad\qquad\qquad\qquad\qquad\qquad\qquad$\phi_h^{n+1}\leftarrow\hat{\phi}_h^{n+1}-\eta^3 \varrho^3 \hat{\phi}_h^{n+1}.$ \\
	\qquad Compute the error estimators:\\
	\qquad\qquad\qquad\qquad\qquad\qquad\qquad$Est_{\u}\leftarrow\eta^4 \varrho^4 \hat{\u}_h^{n+1},$ \quad $Est_{\phi}\leftarrow\eta^4 \varrho^4 \hat{\phi}_h^{n+1}$,\\
	\qquad $\mathbf{If}$ $\{|Est_{\u}|\leq\epsilon,|Est_{\u}|\leq\epsilon\}$, then\\
	\qquad\qquad $
	\mathbf{if} \max \left\{\left|Est_{\u}\right|,\left|Est_{\phi}\right|\right\} <\frac{\epsilon}{4},$ \\
	\qquad\qquad\qquad\qquad$\vartheta_{n+1}\leftarrow\min \left\{2,\left(\frac{\epsilon}{\left|Est_{\u}\right|}\right)^{\frac{1}{3}}, \left(\frac{\epsilon}{\left|E s t_{\phi}\right|}\right)^{\frac{1}{3}}\right\},$ \qquad$k_{n+1}\leftarrow\hat{\gamma} \cdot \vartheta_{n+1} \cdot k_n .$ \\
	\qquad\qquad $\mathbf{end}$\\
	\qquad\qquad$\mathbf { if } \frac{\epsilon}{4} \leq \min \left\{\left|E s t_{\u}\right|,\left|E s t_\phi\right|\right\},$ \\
	\qquad\qquad\qquad\qquad$ \vartheta_{n+1}\leftarrow\min \left\{1,\left(\frac{\epsilon}{\left|E s t_{\u}\right|}\right)^{\frac{1}{3}},\left(\frac{\epsilon}{\left|E s t_\phi\right|}\right)^{\frac{1}{3}}\right\},$ \qquad$k_{n+1}\leftarrow\hat{\gamma} \cdot \vartheta_{n+1} \cdot k_n .$\\
	\qquad\qquad $\mathbf{end}$\\
	\qquad\qquad Then set $k_n\leftarrow k_{n+1}$ and $t_{n+2}\leftarrow t_{n+1}+k_{n+1}$.\\
	\qquad $\mathbf{else}$\\
	\qquad\qquad All approximations are rejected. Redo with new stepsize,\\
	\qquad\qquad\qquad\qquad\qquad\qquad\qquad$k_n\leftarrow k_{n}/\hat{\gamma}\cdot\check{\gamma},$\quad $t_{n+1}\leftarrow t_{n}+k_n.$\\
	\qquad $\mathbf{end}$ \\
	\hline
\end{longtable}

\section{Stability and Convergence Analysis}
\label{sec3:anal}

The following inequalities will be frequently employed in subsequent analyses.
\begin{lemma}\label{interface} \cite{20} There exist constants $C_1=C_t^2\tilde{C_t^2}\geq0$ and $C_2=C_p^2\tilde{C_p^2}\geq0$, such that for all $\forall\;\varepsilon>0$,
	\begin{align}\label{lemma1}
		&|c_{\Gamma}(w,z)|\leq\varepsilon\left\|w\right\|_{\mathbf{W}}^2 +\frac{gC_1C_2}{4\varepsilon k_{\min}\nu} \left\|z\right\|_{\mathbf{W}}^2, \quad \forall\;w,z\in\mathbf{W}.
	\end{align}
	Further, we have
	\begin{align}\label{lemma2}
		&|c_{\Gamma}(w,z)|\leq\frac{\varepsilon}{2}\left(\left\|w\right\|_{\mathbf{W}}^2+\left\|z\right\|_{\mathbf{W}}^2\right) +\frac{gC_1}{8\varepsilon\sqrt{\nu Sk_{\min}}} \left(\left\|w\right\|_{f/p}^2+\left\|z\right\|_{f/p}^2\right), \quad \forall\;w,z\in \mathbf{W}.
	\end{align}
\end{lemma}

\begin{lemma} \label{xishu}
	
	As the algorithm encompasses a multitude of coefficients, we consistently denote $\beta$, $\gamma$, and $\sigma$ uniformly in this context within equations (\ref{sjjd}), (\ref{hat}), and (\ref{jjmxs}),
	\begin{align*}
		&\beta_3=1+\frac{\tau_n}{1+\tau_n}+\frac{\tau_n\tau_{n-1}}{1+\tau_{n-1}\left(1+\tau_n\right)},\quad \qquad\quad\;\;\beta_2=\frac{\tau_n^2}{1+\tau_n}+\frac{\tau_n^2\tau_{n-1}\left(1+\tau_{n-1}\left(2+\tau_n\right)\right)} {\left(1+\tau_{n-1}\right)\left(1+\tau_{n-1}\left(1+\tau_n\right)\right)},\displaybreak[1]\\
		&\beta_1=\frac{\tau_n^2\tau_{n-1}^3\left(1+\tau_n\right)} {\left(1+\tau_{n-1}\right)\left(1+\tau_{n-1}\left(1+\tau_n\right)\right)},\quad \qquad\qquad\gamma_3=\frac{\tau_{n}\tau_{n-1}\left(1+\tau_n\right)}{\left(1+2\tau_n\right)\left(1+\tau_{n-1}\left(1+\tau_n\right)\right)},\\
		&\gamma_2=\frac{\tau_n^2\tau_{n-1}\left(1+\tau_n\right)\left(1+\tau_{n-1}\left(2+\tau_n\right)\right)} {\left(1+2\tau_n\right)\left(1+\tau_{n-1}\right)\left(1+\tau_{n-1}\left(1+\tau_n\right)\right)},\quad \gamma_1=\frac{\tau_n^2\tau_{n-1}^3\left(1+\tau_n\right)^2}{\left(1+2\tau_n\right)\left(1+\tau_{n-1}\right)\left(1+\tau_{n-1}\left(1+\tau_n\right)\right)},\\
		&\sigma_3=1,\qquad\quad \sigma_2=\frac{\tau_n\left(1+\tau_{n-1}\left(2+\tau_n\right)\right)} {1+\tau_{n-1}},\qquad\; \sigma_1=\frac{\tau_n\tau_{n-1}^2\left(1+\tau_n\right)}{1+\tau_{n-1}}.
	\end{align*}
	
	When the experimental function is evaluated at $z=\delta w^{n+1}$, by using the continuity and coerciveness of bilinear form $a(\cdot,\cdot)$ \cite{43}, Cauchy-Schwarz inequality, H$\ddot{o}$lder inequality, and Young's inequality, we can derive
	\begin{align}\label{sjx}
		&\frac{1}{k_{n+1}}\left(\frac{1+2\tau_n}{1+\tau_n}\hat{w}^{n+1}-\left(1+\tau_n\right)w^n +\frac{\tau_{n}^2}{1+\tau_n}w^{n-1},\delta w^{n+1}\right)=\frac{1}{k_{n+1}}\left(\mathcal{A}(w_h^{n+1}),\delta w^{n+1}\right)\nonumber\\
		=&\frac{1}{k_{n+1}}\left(\beta_3\delta w^{n+1}-\beta_2\delta w^{n}+\beta_1\delta w^{n-1}, \delta w^{n+1}\right)\nonumber\\
		=&\frac{\beta_3}{k_{n+1}}\left(\delta w^{n+1},\delta w^{n+1}\right) -\frac{\beta_2}{k_{n+1}}\left(\delta w^{n},\delta w^{n+1}\right) +\frac{\beta_1}{k_{n+1}}\left(\delta w^{n-1}, \delta w^{n+1}\right)\nonumber\\
		\geq&\frac{\beta_3}{k_{n+1}}\left\|\delta w^{n+1}\right\|^2 -\frac{\beta_2}{2k_{n+1}}\left\|\delta w^{n}\right\|^2 -\frac{\beta_2}{2k_{n+1}}\left\|\delta w^{n+1}\right\|^2-\frac{\beta_1}{2k_{n+1}}\left\|\delta w^{n-1}\right\|^2 -\frac{\beta_1}{2k_{n+1}}\left\|\delta w^{n+1}\right\|^2\nonumber\\
		=&\frac{\left(\beta_3-\beta_2-\beta_1\right)}{k_{n+1}}\left\|\delta w^{n+1}\right\|^2 +\frac{\beta_2+\beta_1}{2k_{n+1}}\left(\left\|\delta w^{n+1}\right\|^2-\left\|\delta w^{n}\right\|^2\right)+\frac{\beta_1}{2k_{n+1}}\left(\left\|\delta w^{n}\right\|^2-\left\|\delta w^{n-1}\right\|^2\right),
	\end{align}
	and
	\begin{align}\label{sxx}
		&a\left(\hat{w}^{n+1},\delta w^{n+1}\right)=a\left(w^{n+1}+\gamma_3\delta w^{n+1} -\gamma_2\delta w^{n} +\gamma_1\delta w^{n-1},\delta w^{n+1}\right)\nonumber\\
		=&a\left(w^{n+1},\delta w^{n+1}\right)+\gamma_3a\left(\delta w^{n+1},\delta w^{n+1}\right) -\gamma_2a\left(\delta w^{n},\delta w^{n+1}\right)+\gamma_1a\left(\delta w_{h}^{n-1},\delta w^{n+1}\right)\nonumber\\
		\geq&\frac{1}{2}\left\|w^{n+1}\right\|^2-\frac{1}{2}\left\|w^{n}\right\|^2+\frac{1}{2}\left\|\delta w^{n+1}\right\|^2 +\gamma_3\left\|\delta w^{n}\right\|^2-\frac{\gamma_2}{2}\left\|\delta w^{n}\right\|^2 \nonumber\\
		&-\frac{\gamma_2}{2}\left\|\delta w^{n+1}\right\|^2 -\frac{\gamma_1}{2}\left\|\delta w^{n-1}\right\|^2 -\frac{\gamma_1}{2}\left\|\delta w^{n+1}\right\|^2\nonumber\\
		=&\frac{1}{2}\left\|w^{n+1}\right\|^2-\frac{1}{2}\left\|w^{n}\right\|^2 +\frac{1+2\gamma_3-2\gamma_2-2\gamma_1}{2}\left\|\delta w^{n+1}\right\|^2 \nonumber\\
		&+\frac{\gamma_2+\gamma_1}{2}\left(\left\|\delta w^{n+1}\right\|^2-\left\|\delta w^{n}\right\|^2\right) +\frac{\gamma_1}{2}\left(\left\|\delta w^{n}\right\|^2-\left\|\delta w^{n-1}\right\|^2\right).
	\end{align}
\end{lemma}

In this section, under a time step restriction of the form,
\begin{align}\label{kappa}
	&\kappa:=\frac{gC_1}{20\sqrt{\nu Sk_{\min}}},
\end{align}
where $C_1$ is the constant defined in Lemma \ref{interface}, we prove the stability over bounded time intervals $[0,T]$ of the Algorithm 1. It is also possible to prove stability under the alternate condition $\frac{k_{n+1}}{h}\leq C$ (physical parameters), where $C$ has a different dependency than in (\ref{kappa}).

In addition, we assume condition $\tau_n\approx\tau_{n-1}$ and investigate the scenario where condition $\tau_n$ falls within the range of $(0,1.0315]$. Now we start to investigate the stability and error estimation of the variable time step BDF2-TF algorithm for Stokes-Darcy model.

\subsection{Stability of the variable time step BDF2-TF algorithm}

\begin{theorem}
	Define symbols
	\begin{align}\label{w1}
		E^{n+1}= &\frac{\beta_1+\beta_2}{2}\left\|\delta\u_{h}^{n+1}\right\|_f^2 +\frac{\beta_1}{2}\left\|\delta\u_{h}^{n}\right\|_f^2 +\frac{gS\left(\beta_1+\beta_2\right)}{2}\left\|\delta\phi_{h}^{n+1}\right\|_p^2 +\frac{gS\beta_1}{2}\left\|\delta\phi_{h}^{n}\right\|_p^2,\\
		F^{n+1}=
		&\frac{\nu}{2}\left\|\u_h^{n+1}\right\|_{\mathbf{H}_f}^2 +\frac{\nu(\gamma_1+\gamma_2)}{2}\left\|\delta\u_{h}^{n+1}\right\|_{\mathbf{H}_f}^2 +\frac{\nu\gamma_1}{2}\left\|\delta\u_{h}^{n}\right\|_{\mathbf{H}_f}^2\nonumber\\
		&+\frac{g}{2}\left\|\textbf{K}^{\frac{1}{2}}\nabla\phi_h^{n+1}\right\|_p^2
		+\frac{g(\gamma_1+\gamma_2)}{2}\left\|\textbf{K}^{\frac{1}{2}}\nabla\delta\phi_{h}^{n+1}\right\|_p^2 +\frac{g\gamma_1}{2}\left\|\textbf{K}^{\frac{1}{2}}\nabla\delta\phi_{h}^{n}\right\|_p^2.
	\end{align}
	The decoupled scheme is stable in finite time and there holds
	\begin{align}\label{w0}
		&E^{N}+k_{n+1}F^{N} +\sum\limits_{n=2}^{N-1}\left(\left\|\delta\u_{h}^{n+1}\right\|_f^2 +gS\left\|\delta\phi_{h}^{n+1}\right\|_p^2\right)\nonumber\\
		&+\sum\limits_{n=2}^{N-1}k_{n+1}\left(\nu\left\|\delta\u_{h}^{n+1}\right\|_{\mathbf{H}_f}^2
		+g\left\|\textbf{K}^{\frac{1}{2}}\nabla\delta\phi_{h}^{n+1}\right\|_p^2\right)\nonumber\\
		\leq&C(T)\left\{E^2+k_{n+1}F^{2}+\sum\limits_{n=2}^{N-1}k_{n+1}\left\|\mathbf{f}_1^{n+1}\right\|_f^2 +\sum\limits_{n=2}^{N-1}k_{n+1}\left\|f_2^{n+1}\right\|_p^2 \right\}.
	\end{align}
	with $C(T)\approx\exp\left(\sum\limits_{n=2}^{N-1}k_{n+1}\frac{\kappa}{1-\kappa \left(\sigma_3+\sigma_2+\sigma_1\right)}\right)$.
\end{theorem}

\begin{proof}
	In (\ref{s1})-(\ref{s2}), we set $\vv_h=k_{n+1}\delta\u_{h}^{n+1}$ and $\varphi_h=k_{n+1}\delta\phi_{h}^{n+1}$, and add them together
	\begin{align}\label{w2}
		&\left(\mathcal{A}(\u_h^{n+1}),\delta\u_{h}^{n+1}\right) +\left(\mathcal{A}(\phi_h^{n+1}),\delta\phi_{h}^{n+1}\right) +k_{n+1}a_f\left(\hat{\u}_h^{n+1},\delta\u_{h}^{n+1}\right) +k_{n+1}a_p\left(\hat{\phi}_h^{n+1},\delta\phi_{h}^{n+1}\right)\nonumber\\
		=&k_{n+1}\left(\mathbf{f}_1^{n+1},\delta\u_{h}^{n+1}\right)+ gk_{n+1}\left(f_2^{n+1},\delta\phi_{h}^{n+1}\right) +k_{n+1}c_{\Gamma}\left(\delta\u_{h}^{n+1},\delta\phi_{h}^{n+1};\u_{h,\sigma}^n,\phi_{h,\sigma}^n\right).
	\end{align}
	Through the inequations of $\mathcal{A}(w_h^{n+1})$ and $\hat{w}_h^{n+1}$ in Lemma \ref{xishu}, and combining (\ref{w2}), we obtain
	\begin{align}\label{w3}
		&\left(\beta_3-\beta_2-\beta_1\right)\left\|\delta\u_{h}^{n+1}\right\|_f^2 +gS\left(\beta_3-\beta_2-\beta_1\right)\left\|\delta\phi_{h}^{n+1}\right\|_p^2\nonumber\\
		&+\frac{\beta_1+\beta_2}{2}\left(\left\|\delta\u_{h}^{n+1}\right\|_f^2-\left\|\delta\u_{h}^{n}\right\|_f^2\right) +\frac{gS\left(\beta_1+\beta_2\right)}{2}\left(\left\|\delta\phi_{h}^{n+1}\right\|_p^2-\left\|\delta\phi_{h}^{n}\right\|_p^2\right)\nonumber\\
		&+\frac{\beta_1}{2}\left(\left\|\delta\u_{h}^{n}\right\|_f^2-\left\|\delta\u_{h}^{n-1}\right\|_f^2\right) +\frac{gS\beta_1}{2}\left(\left\|\delta\phi_{h}^{n}\right\|_p^2-\left\|\delta\phi_{h}^{n-1}\right\|_p^2\right)\nonumber\\
		&+\frac{\nu k_{n+1}}{2}\left\|\u_h^{n+1}\right\|_{\mathbf{H}_f}^2-\frac{\nu k_{n+1}}{2}\left\|\u_h^{n}\right\|_{\mathbf{H}_f}^2 +\frac{gk_{n+1}}{2}\left\|\textbf{K}^{\frac{1}{2}}\nabla\phi_h^{n+1}\right\|_p^2-\frac{gk_{n+1}}{2}\left\|\textbf{K}^{\frac{1}{2}}\nabla\phi_h^{n}\right\|_p^2\nonumber\\
		&+\frac{k_{n+1}\left(1+2\gamma_3-2\gamma_2-2\gamma_1\right)}{2}\left(\nu \left\|\delta\u_{h}^{n+1}\right\|_{\mathbf{H}_f}^2 +g\left\|\textbf{K}^{\frac{1}{2}}\nabla\delta\phi_{h}^{n+1})\right\|_p^2\right)\nonumber\\
		&+\frac{k_{n+1}\left(\gamma_1+\gamma_2\right)}{2}\left(\nu\left(\left\|\delta\u_{h}^{n+1}\right\|_{\mathbf{H}_f}^2-\left\|\delta\u_{h}^{n}\right\|_{\mathbf{H}_f}^2\right) +g \left(\left\|\textbf{K}^{\frac{1}{2}}\nabla\delta\phi_{h}^{n+1}\right\|_p^2-\left\|\textbf{K}^{\frac{1}{2}}\nabla\delta\phi_{h}^{n}\right\|_p^2\right)\right)\nonumber\\
		&+\frac{k_{n+1}\gamma_1}{2}\left(\nu\left(\left\|\delta\u_{h}^{n}\right\|_{\mathbf{H}_f}^2-\left\|\delta\u_{h}^{n-1}\right\|_{\mathbf{H}_f}^2\right) +g\left(\left\|\textbf{K}^{\frac{1}{2}}\nabla\delta\phi_{h}^{n}\right\|_p^2-\left\|\textbf{K}^{\frac{1}{2}}\nabla\delta\phi_{h}^{n-1}\right\|_p^2\right)\right)\nonumber\\
		\leq&k_{n+1}\left(\mathbf{f}_1^{n+1},\delta\u_{h}^{n+1}\right)+ gk_{n+1}\left(f_2^{n+1},\delta\phi_{h}^{n+1}\right) +k_{n+1}c_{\Gamma}\left(\delta\u_{h}^{n+1},\delta\phi_{h}^{n+1};\u_{h,\sigma}^n,\phi_{h,\sigma}^n\right).
	\end{align}
	Then we can rearrange the equality
	\begin{align}\label{w4}
		&E^{n+1}-E^{n} +\left(\beta_3-\beta_2-\beta_1\right)\left\|\delta\u_{h}^{n+1}\right\|_f^2 +gS\left(\beta_3-\beta_2-\beta_1\right)\left\|\delta\phi_{h}^{n+1}\right\|_p^2 \nonumber\\
		&+k_{n+1}F^{n+1}-k_{n+1}F^{n}+\frac{ k_{n+1}\left(1+2\gamma_3-2\gamma_2-2\gamma_1\right)}{2}\left(\nu\left\|\delta\u_{h}^{n+1}\right\|_{\mathbf{H}_f}^2 +g \left\|\textbf{K}^{\frac{1}{2}}\nabla\delta\phi_{h}^{n+1}\right\|_p^2\right)\nonumber\\
		\leq&k_{n+1}\left(\mathbf{f}_1^{n+1},\delta\u_{h}^{n+1}\right)+ gk_{n+1}\left(f_2^{n+1},\delta\phi_{h}^{n+1}\right) +k_{n+1}c_{\Gamma}\left(\delta\u_{h}^{n+1},\delta\phi_{h}^{n+1};\u_{h,\sigma}^n,\phi_{h,\sigma}^n\right).
	\end{align}
	
	Next we deal with the right-hand side of the inequality, where the external force terms are given by the using Young's and H$\ddot{o}$lder's inequalities
	\begin{align}\label{w5}
		&\left(\mathbf{f}_1^{n+1},\delta\u_{h}^{n+1}\right)+ g\left(f_2^{n+1},\delta\phi_{h}^{n+1}\right)\nonumber\\
		\leq&\frac{10C_p^2}{\nu}\left\|\mathbf{f}_1^{n+1}\right\|_f^2 +\frac{10g\widetilde{C}_p^2}{k_{\min}}\left\|f_2^{n+1}\right\|_p^2 +\frac{\nu }{40}\left\|\delta\u_{h}^{n+1}\right\|_{\mathbf{H}_f}^2 +\frac{g}{40}\left\|\textbf{K}^{\frac{1}{2}}\nabla\delta\phi_{h}^{n+1}\right\|_p^2.
	\end{align}
	The remains of right-hand side in (\ref{w4}) are bounded by (\ref{lemma2})
	\begin{align}\label{w6}
		&c_{\Gamma}\left(\delta\u_{h}^{n+1},\delta\phi_{h}^{n+1};\u_{h,\sigma}^n,\phi_{h,\sigma}^n\right)\nonumber\\
		=&c_{\Gamma}\left(\delta\u_{h}^{n+1},\delta\phi_{h}^{n+1}; -\sigma_3\u_h^n-\sigma_2\delta\u_{h}^{n} +\sigma_1\delta\u_{h}^{n-1},-\sigma_3\phi_h^n-\sigma_2\delta\phi_{h}^{n} +\sigma_1\delta\phi_{h}^{n-1}\right)\nonumber\\
		\leq&\frac{\sigma_3}{40}\left(\nu\left\|\delta\u_{h}^{n+1}\right\|_{\mathbf{H}_f}^2 +g\left\|\textbf{K}^{\frac{1}{2}}\nabla\delta\phi_h^{n+1}\right\|_p^2 +\nu\left\|\u_{h}^{n}\right\|_{\mathbf{H}_f}^2 +g\left\|\textbf{K}^{\frac{1}{2}}\nabla\phi_h^n\right\|_p^2\right)\nonumber\\
		&+\frac{gC_1 \sigma_3}{20\sqrt{\nu Sk_{\min}}}\left(\left\|\delta\u_{h}^{n+1}\right\|_{f}^2 +gS\left\|\delta\phi_h^{n+1}\right\|_p^2 +\left\|\u_{h}^{n}\right\|_{f}^2 +gS\left\|\phi_h^n\right\|_p^2\right)\nonumber\\
		&+\frac{\sigma_2}{40}\left(\nu\left\|\delta\u_{h}^{n+1}\right\|_{\mathbf{H}_f}^2 +g\left\|\textbf{K}^{\frac{1}{2}}\nabla\delta\phi_h^{n+1}\right\|_p^2 +\nu\left\|\delta\u_{h}^{n}\right\|_{\mathbf{H}_f}^2 +g\left\|\textbf{K}^{\frac{1}{2}}\nabla\delta\phi_h^n\right\|_p^2\right)\nonumber\\
		&+\frac{gC_1 \sigma_2}{20\sqrt{\nu Sk_{\min}}}\left(\left\|\delta\u_{h}^{n+1}\right\|_{f}^2 +gS\left\|\delta\phi_h^{n+1}\right\|_p^2 +\left\|\delta\u_{h}^{n}\right\|_{f}^2 +gS\left\|\delta\phi_h^n\right\|_p^2\right)\nonumber\\
		&+\frac{\sigma_1}{40}\left(\nu\left\|\delta\u_{h}^{n+1}\right\|_{\mathbf{H}_f}^2 +g\left\|\textbf{K}^{\frac{1}{2}}\nabla\delta\phi_h^{n+1}\right\|_p^2 +\nu\left\|\delta\u_{h}^{n-1}\right\|_{\mathbf{H}_f}^2 +g\left\|\textbf{K}^{\frac{1}{2}}\nabla\delta\phi_h^{n-1}\right\|_p^2\right)\nonumber\\
		&+\frac{gC_1 \sigma_1}{20\sqrt{\nu Sk_{\min}}}\left(\left\|\delta\u_{h}^{n+1}\right\|_{f}^2 +gS\left\|\delta\phi_h^{n+1}\right\|_p^2 +\left\|\delta\u_{h}^{n-1}\right\|_{f}^2 +gS\left\|\delta\phi_h^{n-1}\right\|_p^2\right).
	\end{align}
	Inserting (\ref{w6}) and (\ref{w5}) to (\ref{w4}), and summing over $n=2,3,\ldots,N-1$, we arrive at
	\begin{align}\label{w8}
		&E^{N}+k_{n+1}F^{N} +\left(\beta_3-\beta_2-\beta_1\right)\sum\limits_{n=2}^{N-1}\left(\left\|\delta\u_{h}^{n+1}\right\|_f^2 +gS\left\|\delta\phi_{h}^{n+1}\right\|_p^2\right)\nonumber\\
		&+\frac{\nu \left(19+40\gamma_3-40\gamma_2-40\gamma_1-\sigma_1-\sigma_2-\sigma_3\right)}{40}\sum\limits_{n=2}^{N-1}k_{n+1}\left(\left\|\delta\u_{h}^{n+1}\right\|_{\mathbf{H}_f}^2 +g\left\|\textbf{K}^{\frac{1}{2}}\nabla\delta\phi_{h}^{n+1}\right\|_p^2\right)\nonumber\\
		\leq&E^2+\frac{10C_p^2}{\nu}\sum\limits_{n=2}^{N-1}k_{n+1}\left\|\mathbf{f}_1^{n+1}\right\|_f^2 +\frac{gC_1\left(\sigma_3+\sigma_2+\sigma_1\right)}{20\sqrt{\nu Sk_{\min}}}\sum\limits_{n=2}^{N-1}k_{n+1}\left(\left\|\delta\u_{h}^{n+1}\right\|_{f}^2 +gS\left\|\delta\phi_{h}^{n+1}\right\|_{p}^2\right)\nonumber\\
		&+k_{n+1}F^{2}+\frac{10g\widetilde{C}_p^2}{k_{\min}}\sum\limits_{n=2}^{N-1}k_{n+1}\left\|f_2^{n+1}\right\|_p^2 +\frac{gC_1 \sigma_3}{20\sqrt{\nu Sk_{\min}}}\sum\limits_{n=2}^{N-1}k_{n+1}\left(\left\|\u_h^{n}\right\|_{f}^2+gS\left\|\phi_h^{n}\right\|_{f}^2\right)\displaybreak[1]\nonumber\\
		& +\frac{gC_1}{20\sqrt{\nu Sk_{\min}}}\sum\limits_{n=2}^{N-1}k_{n+1}\left(\sigma_2\left(\left\|\delta\u_{h}^{n}\right\|_{f}^2 +gS\left\|\delta\phi_{h}^{n}\right\|_{p}^2\right) +\sigma_1\left(\left\|\delta\u_{h}^{n-1}\right\|_{f}^2 +gS\left\|\delta\phi_{h}^{n-1}\right\|_{p}^2\right)\right).
	\end{align}
	Setting
	$$\kappa:=\frac{gC_1}{20\sqrt{\nu Sk_{\min}}}.$$
	Then, suppose that time step coefficients satisfies
	$$\left(\sigma_3+\sigma_2+\sigma_1\right)\kappa<1.$$
	By using the discrete Gronwall inequality, we get the stability result (\ref{w0}).
\end{proof}

\subsection{Error analysis}

We assume that the solution of the Stokes-Darcy problem follow the subsequent regularity: $\left(\u(t),\phi(t)\right)\in \left(\mathbf{H}^3(\Omega_f)^d,H^3(\Omega_p)\right)$, $\left(\u_t(t),\phi_t(t)\right)\in \left(\mathbf{H}^2(\Omega_f)^d\right.$, $\left.H^2(\Omega_p)\right)$ and $\left(\u_{tt}(t),\phi_{tt}(t)\right)\in \left(L^2(\Omega_f)^d,L^2(\Omega_p)\right)$, define the linear projection operator $P_h^{\Phi}$: $\left(\u(t_n),p(t_n),\phi(t_n)\right)\in \left(\mathbf{H}_f,Q,H_p\right)\rightarrow\left(\tilde{\u}^n,\tilde{\phi}^n,\tilde{p}^n\right) \in\left(\mathbf{H}_{fh},Q_h,H_{ph}\right)$, $\forall\; t\in [0,T]$, $\forall\; z(t_n)\in \mathbf{W}_{h},$ and $\forall\;q(t_n)\in Q_h$ by:
\begin{align}\label{P1}
	&a(\tilde{w}^n,z_h)+b(z_h,\tilde{p}^n)
	=a(w(t_n),z_h)+b(z_h,p(t_n)),\nonumber\\
	&b(\tilde{w}^n,q_h)=0.
\end{align}
Furthermore, it is essential to establish the error function:
\begin{align*}
	&\e_{\u}^{n+1}=\u^{n+1}-\u_h^{n+1}=\u^{n+1}-\tilde{\u}^{n+1} +\tilde{\u}^{n+1}-\u_h^{n+1}=\xi_{\u}^{n+1}+\eta_{\u}^{n+1},\\
	&\e_{p}^{n+1}=p^{n+1}-p_h^{n+1}=p^{n+1}-\tilde{p}^{n+1} +\tilde{p}^{n+1}-p_h^{n+1}=\xi_{p}^{n+1}+\eta_{p}^{n+1},\\
	&\e_{\phi}^{n+1}=\phi^{n+1}-\phi_h^{n+1}=\phi^{n+1}-\tilde{\phi}^{n+1} +\tilde{\phi}^{n+1}-\phi_h^{n+1}=\xi_{\phi}^{n+1}+\eta_{\phi}^{n+1}.
\end{align*}
Based on the properties of approximation, we can infer $\|\xi_{\u}^{n+1}\|_{L^2(\Omega_f)}^2+\|\xi_{\phi}^{n+1}\|_{L^2(\Omega_p)}^2\leq Ch^6$, $\|\xi_{\u}^{n+1}\|_{\mathbf{H}_f}^2+\|\xi_{\phi}^{n+1}\|_{H_p}^2\leq Ch^4$ \cite{3,23}. And we suppose that $\eta_{\u}^0=\eta_{\u}^1=\eta_{\u}^2=0$, $\eta_{\phi}^0=\eta_{\phi}^1=\eta_{\phi}^2=0$.

By (\ref{wf}) and (\ref{P1}), for $\forall\;\left(z_h,q_h\right)\in\left(\mathbf{W}_{h},Q_h\right)$, we have
\begin{align}\label{P3}
	&\left(\frac{\mathcal{A}(\tilde{\u}^{n+1})}{k_{n+1}},\vv_h\right) +a_f\left(\tilde{\hat{\u}}^{n+1},\vv_h\right)+b\left(\vv_h,\tilde{\hat{p}}^{n+1}\right)\nonumber\\
	=&-\left(\varpi_{f,t}^{n+1},\vv_h\right)+\left(\mathbf{f}_1^{n+1},\vv_h\right) -g\int_{\Gamma}\tilde{\hat{\phi}}^{n+1}\vv_h\cdot\n_f,\\
	&b\left(\tilde{\u}^{n+1},q_h\right)=0,\label{P30}
\end{align}
\begin{align}\label{P31}
	&gS\left(\frac{\mathcal{A}(\tilde{\phi}^{n+1})}{k_{n+1}},\varphi_h\right) +a_p\left(\tilde{\hat{\phi}}^{n+1},\varphi_h\right)
	=-gS\left(\varpi_{p,t}^{n+1},\varphi_h\right)+g\left(f_2^{n+1},\varphi_h\right) +g\int_{\Gamma}\varphi_h\tilde{\hat{\u}}^{n+1}\cdot\n_f,
\end{align}
where $\varpi_{w,t}^{n+1}$ is defined by
\begin{align*}
	\varpi_{w,t}^{n+1}&=\frac{\mathcal{A}(\tilde{w}^{n+1})}{k_{n+1}}-w_t(t_{n+1}) =\left[\frac{\mathcal{A}(\tilde{w}^{n+1})}{k_{n+1}}-\frac{\mathcal{A}(w(t_{n+1}))}{k_{n+1}}\right] +\left[\frac{\mathcal{A}(w(t_{n+1}))}{k_{n+1}}-w_t(t_{n+1})\right].
\end{align*}

To establish the error estimates of the BDF2-TF algorithm, we initially present the consistency error associated with this algorithm.
\begin{lemma} \label{ab} The following inequalities hold
	\begin{align}\label{pro3}
		&\left\|\frac{\mathcal{A}(\u(t_{n+1}))}{k_{n+1}}-\u_t(t_{n+1})\right\|_f^2 \leq Ck_{n+1}^5\int_{t_{n-2}}^{t_{n+1}}\left\|\u_{tttt}(t)\right\|_f^2dt,\\
		&\left\|\frac{\mathcal{A}(\tilde{\u}^{n+1})}{k_{n+1}}-\frac{\mathcal{A}(\u(t_{n+1}))}{k_{n+1}}\right\|_f^2 \leq \frac{C}{k_{n+1}}\int_{t_{n-2}}^{t_{n+1}}\left\|(P_h^{\u}-I)\u_{t}\right\|_f^2dt,\label{pro31}\\
		&\left\|\tilde{\hat{\u}}^{n+1}-\tilde{\u}_{\sigma}^{n}\right\|_{H_f}^2\leq Ck_{n+1}^5\int_{t_{n-2}}^{t_{n+1}}\left\|\u_{ttt}\right\|_{\mathbf{H}_f}^2dt.\label{pro32}
	\end{align}
\end{lemma}

\begin{proof} The proof is similar to the Lemma 2 in \cite{44}.\end{proof}

\begin{theorem} \label{error}
	For any $0<t_N=T<\infty$, assuming the true solution exhibits smoothness, the initial approximations are adequately precise, and that the time step coefficients fulfills
	$$\left(\sigma_3+\sigma_2+\sigma_1\right)\kappa<1,$$
	then the subsequent estimation for the error at the time increment is valid $(below\;\;\delta\e_{w}^{n}:=\e_{w}^{n}-\e_{w}^{n-1})$:
	
	\begin{align}\label{e0}
		&E_{\e}^{N}+k_{n+1}F_{\e}^{N}+\sum\limits_{n=2}^{N-1}\left(\left\|\delta\e_{\u}^{n+1}\right\|_f^2+ gS\left\|\delta\e_{\phi}^{n+1}\right\|_p^2\right) \nonumber\\
		&+\sum\limits_{n=2}^{N-1}k_{n+1}\left(\nu\left\|\delta\e_{\u}^{n+1}\right\|_{\mathbf{H}_f}^2 +g\left\|\textbf{K}^{\frac{1}{2}}\nabla\delta\e_{\phi}^{n+1}\right\|_p^2\right)\nonumber\\
		\leq&C(k_{n+1}^6+h^6).
	\end{align}
	In the present and subsequent contexts, $C$ denotes a generic positive constant that relies on the given data $(\nu,g,S,k_{\min})$, which may exhibit distinct values in its various instances.
	
\end{theorem}

\begin{proof}
	By subtracting equations (\ref{P3})-(\ref{P31}) from (\ref{s1})-(\ref{s2}), we obtain the subsequent error equations:
	\begin{align}\label{e1}
		&\left(\frac{\mathcal{A}(\eta_{\u}^{n+1})}{k_{n+1}},\vv_h\right) +a_f\left(\hat{\eta}_{\u}^{n+1},\vv_h\right) +b\left(\vv_h,\hat{\eta}_{p}^{n+1}\right)\nonumber\\
		&=-\left(\varpi_{f,t}^{n+1},\vv_h\right) -g\int_{\Gamma}\left(\tilde{\hat{\phi}}^{n+1}-\tilde{\phi}_{\sigma}^n\right)\vv_h\cdot\n_f -g\int_{\Gamma}\left(\tilde{\phi}_{\sigma}^n-\phi_{h,\sigma}^n\right)\vv_h\cdot\n_f,\nonumber\\
		&b\left(\delta\eta_{\u}^{n+1},q_h\right)=0,\\
		&gS\left(\frac{\mathcal{A}(\eta_{\phi}^{n+1})}{k_{n+1}},\varphi_h\right) +a_p\left(\hat{\eta}_{\phi}^{n+1},\vv_h\right)\nonumber\\
		&=-gS\left(\varpi_{p,t}^{n+1},\varphi_h\right) +g\int_{\Gamma}\varphi_h\left(\tilde{\hat{\u}}^{n+1}-\tilde{\u}_{\sigma}^n\right)\cdot\n_f +g\int_{\Gamma}\varphi_h\left(\tilde{\u}_{\sigma}^n-\u_{h,\sigma}^n\right)\cdot\n_f.\label{De}
	\end{align}
	Taking $\vv_h=k_{n+1}\delta\eta_{\u}^{n+1}$ and $\varphi_h=k_{n+1}\delta\eta_{\phi}^{n+1}$ in (\ref{e1}) and (\ref{De}), respectively, and combining them together, we obtain
	\begin{align}\label{e2}
		&\left(\beta_3-\beta_2-\beta_1\right)\left\|\delta\eta_{\u}^{n+1}\right\|_f^2 +gS\left(\beta_3-\beta_2-\beta_1\right)\left\|\delta\eta_{\phi}^{n+1}\right\|_p^2\nonumber\\
		&+\frac{\beta_1+\beta_2}{2}\left(\left\|\delta\eta_{\u}^{n+1}\right\|_f^2 -\left\|\delta\eta_{\u}^{n}\right\|_f^2\right) +\frac{\beta_1}{2}\left(\left\|\delta\eta_{\u}^{n}\right\|_f^2 -\left\|\delta\eta_{\u}^{n-1}\right\|_f^2\right)\displaybreak[1]\nonumber\\
		&+\frac{gS\left(\beta_1+\beta_2\right)}{2}\left(\left\|\delta\eta_{\phi}^{n+1}\right\|_p^2 -\left\|\delta\eta_{\phi}^{n}\right\|_p^2\right) +\frac{gS\beta_1}{2}\left(\left\|\delta\eta_{\phi}^{n}\right\|_p^2 -\left\|\delta\eta_{\phi}^{n-1}\right\|_p^2\right)\nonumber\\
		&+\frac{\nu k_{n+1}}{2}\left\|\eta_{\u}^{n+1}\right\|_{\mathbf{H}_f}^2 -\frac{\nu k_{n+1}}{2}\left\|\eta_{\u}^{n}\right\|_{\mathbf{H}_f}^2 +\frac{gk_{n+1}}{2}\left\|\textbf{K}^{\frac{1}{2}}\nabla\eta_{\phi}^{n+1}\right\|_p^2 -\frac{gk_{n+1}}{2}\left\|\textbf{K}^{\frac{1}{2}}\nabla\eta_{\phi}^{n}\right\|_p^2\nonumber\\
		&+\frac{ k_{n+1}\left(1+2\gamma_3-2\gamma_2-2\gamma_1\right)}{2}\left(\nu\left\|\delta\eta_{\u}^{n+1}\right\|_{\mathbf{H}_f}^2 +g \left\|\textbf{K}^{\frac{1}{2}}\nabla\delta\eta_{\phi}^{n+1}\right\|_p^2\right)\nonumber\\
		&+\frac{ k_{n+1}\left(\gamma_1+\gamma_2\right)}{2}\left(\nu\left(\left\|\delta\eta_{\u}^{n+1}\right\|_{\mathbf{H}_f}^2 -\left\|\delta\eta_{\u}^{n}\right\|_{\mathbf{H}_f}^2\right)+g\left(\left\|\textbf{K}^{\frac{1}{2}}\nabla\delta\eta_{\phi}^{n+1}\right\|_p^2 -\left\|\textbf{K}^{\frac{1}{2}}\nabla\delta\eta_{\phi}^{n}\right\|_p^2\right)\right)\nonumber\\ 
		&+\frac{ k_{n+1}\gamma_1}{2}\left(\nu\left(\left\|\delta\eta_{\u}^{n}\right\|_{\mathbf{H}_f}^2-\left\|\delta\eta_{\u}^{n-1}\right\|_{\mathbf{H}_f}^2\right) +g \left(\left\|\textbf{K}^{\frac{1}{2}}\nabla\delta\eta_{\phi}^{n}\right\|_p^2 -\left\|\textbf{K}^{\frac{1}{2}}\nabla\delta\eta_{\phi}^{n-1}\right\|_p^2\right)\right)\nonumber\\
		=&-k_{n+1}\left(\varpi_{f,t}^{n+1},\delta\eta_{\u}^{n+1}\right) +k_{n+1}c_{\Gamma}\left(\delta\eta_{\u}^{n+1},\delta\eta_{\phi}^{n+1};\tilde{\hat{\u}}^{n+1}-\tilde{\u}_{\sigma}^n,\tilde{\hat{\phi}}^{n+1}-\tilde{\phi}_{\sigma}^n\right)\nonumber\\
		&-gSk_{n+1}\left(\varpi_{p,t}^{n+1},\delta\eta_{\phi}^{n+1}\right) +k_{n+1}c_{\Gamma}\left(\delta\eta_{\u}^{n+1},\delta\eta_{\phi}^{n+1};\tilde{\u}_{\sigma}^{n+1}-\u_{h,\sigma}^n,\tilde{\phi}_{\sigma}^{n+1}-\phi_{h,\sigma}^n\right).
	\end{align}
	
	The error estimates $E_{\eta}^{n+1}$ and $F_{\eta}^{n+1}$ are utilized, along with equations (\ref{sjx}) and (\ref{sxx}), yielding the following result
	\begin{align}\label{e3}
		&E_{\eta}^{n+1}-E_{\eta}^{n}+\left(\beta_3-\beta_2-\beta_1\right)\left\|\delta\eta_{\u}^{n+1}\right\|_f^2 +gS\left(\beta_3-\beta_2-\beta_1\right)\left\|\delta\eta_{\phi}^{n+1}\right\|_p^2\nonumber\\
		&+k_{n+1}F_{\eta}^{n+1}-k_{n+1}F_{\eta}^{n} +\frac{ k_{n+1}\left(1+2\gamma_3-2\gamma_2-2\gamma_1\right)}{2}\left(\nu\left\|\delta\eta_{\u}^{n+1}\right\|_{\mathbf{H}_f}^2 +g\left\|\textbf{K}^{\frac{1}{2}}\nabla\delta\eta_{\phi}^{n+1}\right\|_p^2\right)\nonumber\\
		=&-k_{n+1}\left(\varpi_{f,t}^{n+1},\delta\eta_{\u}^{n+1}\right) +k_{n+1}c_{\Gamma}\left(\delta\eta_{\u}^{n+1},\delta\eta_{\phi}^{n+1};\hat{\tilde{\u}}^{n+1}-\tilde{\u}_{\sigma}^n,\hat{\tilde{\phi}}^{n+1}-\tilde{\phi}_{\sigma}^n\right)\nonumber\\ &-gSk_{n+1}\left(\varpi_{p,t}^{n+1},\delta\eta_{\phi}^{n+1}\right) +k_{n+1}c_{\Gamma}\left(\delta\eta_{\u}^{n+1},\delta\eta_{\phi}^{n+1};\tilde{\u}_{\sigma}^{n}-\u_{h,\sigma}^n,\tilde{\phi}_{\sigma}^{n}-\phi_{h,\sigma}^n\right).
	\end{align}
	Terms about $\varpi_{f,t}^{n+1}$ and $\varpi_{p,t}^{n+1}$ of the Right-hand term in (\ref{e3}) is bounded by Young's, Poincar$\acute{e}$'s, and H$\ddot{o}$lder's inequalities,
	\begin{align}\label{e4}
		&-\left(\varpi_{f,t}^{n+1},\delta\eta_{\u}^{n+1}\right) -gS\left(\varpi_{p,t}^{n+1},\delta\eta_{\phi}^{n+1}\right)\nonumber\\
		\leq&\frac{\nu}{40}\left\|\delta\eta_{\u}^{n+1}\right\|_{\mathbf{H}_f}^2 +\frac{g}{40}\left\|\textbf{K}^{\frac{1}{2}}\nabla\delta\eta_{\phi}^{n+1}\right\|_{p}^2 +\frac{10C_p^2}{\nu}\left\|\varpi_{f,t}^{n+1}\right\|_{f}^2 +\frac{10g\tilde{C}_p^2S^2}{k_{\min}}\left\|\varpi_{p,t}^{n+1}\right\|_{p}^2.
	\end{align}
	Using the identical technique employed in (\ref{lemma1}) and using (\ref{pro32}), we can address the first line of the interface term on the right hand side of (\ref{e3})
	\begin{align}\label{e5}
		&c_{\Gamma}\left(\delta\eta_{\u}^{n+1},\delta\eta_{\phi}^{n+1};\tilde{\hat{\u}}^{n+1}-\tilde{\u}_{\sigma}^n,\tilde\hat{{\phi}}^{n+1}-\tilde{\phi}_{\sigma}^n\right)\nonumber\\
		\leq&\frac{\nu}{40}\left\|\delta\eta_{\u}^{n+1}\right\|_{\mathbf{H}_f}^2 +\frac{10gC_1C_2}{k_{\min}}\left\|\tilde{\hat{\u}}^{n+1}-\tilde{\u}_{\sigma}^n\right\|_{\mathbf{H}_f}^2 +\frac{g }{40}\left\|\textbf{K}^{\frac{1}{2}}\nabla\delta\eta_{\phi}^{n+1}\right\|_{p}^2 +\frac{10g^2C_1C_2}{\nu}\left\|\tilde{\hat{\phi}}^{n+1}-\tilde{\phi}_{\sigma}^n\right\|_{H_p}^2
		\nonumber\\
		\leq&\frac{\nu}{40}\left\|\delta\eta_{\u}^{n+1}\right\|_{\mathbf{H}_f}^2 +\frac{g }{40}\left\|\textbf{K}^{\frac{1}{2}}\nabla\delta\eta_{\phi}^{n+1}\right\|_{p}^2
		+Ck_{n+1}^5 \int_{t_{n-2}}^{t_{n+1}}\left(\left\|\u_{ttt}\right\|_{\mathbf{H}_f}^2 +\left\|\phi_{ttt}\right\|_{H_p}^2\right).
	\end{align}
	The second line of the interface term on right-hand side of (\ref{e3}) is bounded by inequality (\ref{lemma2})
	\begin{align}\label{e6}
		&c_{\Gamma}\left(\delta\eta_{\u}^{n+1},\delta\eta_{\phi}^{n+1}; \tilde{\u}_{\sigma}^{n}-\u_{h,\sigma}^n,\tilde{\phi}_{\sigma}^{n}-\phi_{h,\sigma}^n\right)
		=c_{\Gamma}\left(\delta\eta_{\u}^{n+1},\delta\eta_{\phi}^{n+1}; \eta_{\u}^{n},\eta_{\phi}^{n}\right)\nonumber\\
		\leq&\frac{\sigma_3+\sigma_2+\sigma_1}{40}\left(\nu\left\|\delta\eta_{\u}^{n+1}\right\|_{\mathbf{H}_f}^2 +g\left\|\textbf{K}^{\frac{1}{2}}\nabla\delta\eta_{\phi}^{n+1}\right\|_{p}^2\right)\nonumber\\
		&+\frac{gC_1\left(\sigma_3+\sigma_2+\sigma_1\right)}{10\sqrt{\nu Sk_{\min}}}\left(\left\|\delta\eta_{\u}^{n+1}\right\|_{f}^2 +gS\left\|\delta\eta_{\phi}^{n+1}\right\|_{p}^2\right)\nonumber\\
		&+\frac{\sigma_3}{40}\left(\nu\left\|\eta_{\u}^{n}\right\|_{\mathbf{H}_f}^2 +g\left\|\textbf{K}^{\frac{1}{2}}\nabla\eta_{\phi}^{n}\right\|_{p}^2\right) +\frac{gC_1 \sigma_3}{10\sqrt{\nu Sk_{\min}}}\left(\left\|\eta_{\u}^{n}\right\|_{f}^2+gS\left\|\eta_{\phi}^{n}\right\|_{f}^2\right)\nonumber\\
		&+\frac{\sigma_2}{40}\left(\nu\left\|\delta\eta_{\u}^{n}\right\|_{\mathbf{H}_f}^2 +g\left\|\textbf{K}^{\frac{1}{2}}\nabla\delta\eta_{\phi}^{n}\right\|_{p}^2\right) +\frac{gC_1\sigma_2}{10\sqrt{\nu Sk_{\min}}}\left(\left\|\delta\eta_{\u}^{n}\right\|_{f}^2 +gS\left\|\delta\eta_{\phi}^{n}\right\|_{p}^2\right)\nonumber\\
		&+\frac{\sigma_1}{40}\left(\nu\left\|\delta\eta_{\u}^{n-1}\right\|_{\mathbf{H}_f}^2 +g\left\|\textbf{K}^{\frac{1}{2}}\nabla\delta\eta_{\phi}^{n-1}\right\|_{p}^2\right) +\frac{gC_1\sigma_1}{10\sqrt{\nu Sk_{\min}}}\left(\left\|\delta\eta_{\u}^{n-1}\right\|_{f}^2 +gS\left\|\delta\eta_{\phi}^{n-1}\right\|_{f}^2\right).
	\end{align}
	
	Combine the aforementioned bounds, incorporate the initial data, sum from $n=2,\ldots,N-1$ and combine (\ref{e2})-(\ref{e6}), we can yield
	\begin{align}\label{e7}
		&E_{\eta}^{N}+k_{n+1}F_{\eta}^{N}+\left(\beta_3-\beta_2-\beta_1\right)\sum\limits_{n=2}^{N-1}\left\|\delta\eta_{\u}^{n+1}\right\|_f^2 +gS\left(\beta_3-\beta_2-\beta_1\right)\sum\limits_{n=2}^{N-1}\left\|\delta\eta_{\phi}^{n+1}\right\|_p^2\nonumber\\
		&+\frac{\left(18+20\gamma_3-20\gamma_2-20\gamma_1-\sigma_3-\sigma_2-\sigma_1\right)}{40}\sum\limits_{n=2}^{N-1}k_{n+1}\left(\nu\left\|\delta\eta_{\u}^{n+1}\right\|_{\mathbf{H}_f}^2 +g\left\|\textbf{K}^{\frac{1}{2}}\nabla\delta\eta_{\phi}^{n+1}\right\|_p^2\right)\nonumber\\
		\leq&\frac{10C_p^2}{\nu}\sum\limits_{n=2}^{N-1}k_{n+1}\left\|\varpi_{f,t}^{n+1}\right\|_{f}^2 +\frac{gC_1\left(\sigma_3+\sigma_2+\sigma_1\right)}{10\sqrt{\nu Sk_{\min}}}\sum\limits_{n=2}^{N-1}k_{n+1}\left(\left\|\delta\eta_{\u}^{n+1}\right\|_{f}^2 +gS\left\|\delta\eta_{\phi}^{n+1}\right\|_{p}^2\right)\nonumber\\ &+\frac{10g\tilde{C}_p^2S^2}{k_{\min}}\sum\limits_{n=2}^{N-1}k_{n+1}\left\|\varpi_{p,t}^{n+1}\right\|_{p}^2 +Ck_{n+1}^6 \int_{t_{n-2}}^{t_{n+1}}\left(\left\|\u_{ttt}\right\|_{\mathbf{H}_f}^2 +\left\|\phi_{ttt}\right\|_{H_p}^2\right)\nonumber\\
		&+\frac{gC_1 }{10\sqrt{\nu Sk_{\min}}}\sum\limits_{n=2}^{N-1}k_{n+1}\left(\sigma_3\left(\left\|\eta_{\u}^{n}\right\|_{f}^2+gS\left\|\eta_{\phi}^{n}\right\|_{f}^2\right) +\sigma_2\left(\left\|\delta\eta_{\u}^{n}\right\|_{f}^2 +gS\left\|\delta\eta_{\phi}^{n}\right\|_{p}^2\right)\right)\nonumber\\
		&+\frac{gC_1\sigma_1}{10\sqrt{\nu Sk_{\min}}}\sum\limits_{n=2}^{N-1}k_{n+1}\left(\left\|\delta\eta_{\u}^{n-1}\right\|_{f}^2 +gS\left\|\delta\eta_{\phi}^{n-1}\right\|_{f}^2\right).
	\end{align}
	Using inequalities (\ref{pro3}) and (\ref{pro31}) of Lemma \ref{ab}, the $\varpi_{f,t}^{n+1}$ and $\varpi_{p,t}^{n+1}$ in (\ref{e7}) is bounded by
	\begin{align}\label{e31}
		&\frac{10C_p^2}{\nu}\sum\limits_{n=2}^{N-1}\left\|\varpi_{f,t}^{n+1}\right\|_{f}^2 \leq C\int_{t_{n-2}}^{t_{n+1}}\left\|(P_h^{\u}-I)\u_t\right\|_f^2dt +Ck_{n+1}^5\int^{t_{n+1}}_{t_{n-2}}\u_{tttt}^2dt
		\leq C(h^6+k_{n+1}^5).
	\end{align}
	Similarly,
	\begin{align}\label{e32}
		&\frac{10g\tilde{C}_p^2S^2}{k_{\min}}\sum\limits_{n=2}^{N-1}\left\|\varpi_{p,t}^{n+1}\right\|_{p}^2 \leq C\int_{t_{n-2}}^{t_{n+1}}\left\|(P_h^{\phi}-I)\phi_t\right\|_p^2dt +Ck_{n+1}^5\int^{t_{n+1}}_{t_{n-2}}\phi_{tttt}^2dt
		\leq C(h^6+k_{n+1}^5).
	\end{align}
	
	Apply the discrete Gronwall inequality, we obtain the ultimate outcome
	\begin{align}\label{e8}
		&E_{\eta}^{N}+k_{n+1}F_{\eta}^{N}+\left(\beta_3-\beta_2-\beta_1\right)\sum\limits_{n=2}^{N-1}\left\|\delta\eta_{\u}^{n+1}\right\|_f^2 +gS\left(\beta_3-\beta_2-\beta_1\right)\sum\limits_{n=2}^{N-1}\left\|\delta\eta_{\phi}^{n+1}\right\|_p^2\nonumber\\
		&+\frac{\left(18+40\gamma_3-40\gamma_2-40\gamma_1-\sigma_3-\sigma_2-\sigma_1\right)}{40}\sum\limits_{n=2}^{N-1}k_{n+1}\left(\nu\left\|\delta\eta_{\u}^{n+1}\right\|_{\mathbf{H}_f}^2 +g\left\|\textbf{K}^{\frac{1}{2}}\nabla\delta\eta_{\phi}^{n+1}\right\|_p^2\right)\nonumber\\
		\leq&CC(T)(k_{n+1}^6+h^6),
	\end{align}
	with $C(T)\approx\exp\left(\sum\limits_{n=2}^{N-1}k_{n+1}\frac{\kappa}{1-\kappa \left(\sigma_3+\sigma_2+\sigma_1\right)}\right)$, we've done the proof.
\end{proof}

\begin{theorem} The error between the approximate solution $\left\{\u_{h}^{n+1},\phi_{h}^{n+1}\right\}_{n=2}^{N}$ and the exact solution of the variable time step BDF2-TF algorithm satisfies
	\begin{align}\label{th3}
		&\left\|\u^{n+1}-\u_h^{n+1}\right\|_f+\left\|\phi^{n+1}-\phi_h^{n+1}\right\|_p\leq Ch^3+\max\limits_{2\leq n\leq N}k_{n+1}^3.
	\end{align}
\end{theorem}
\begin{proof} Because
	\begin{align}\label{e}
		&\left\|\e\right\|_{2,0}\leq\left\|\eta\right\|_{2,0}+\left\|\xi\right\|_{2,0},
	\end{align}
	by (\ref{P1}), linearity of the projection operator $P_{h}^{\Phi}$ and Theorem \ref{error}, we have (\ref{th3}).
\end{proof}

\section{Extension to adaptive algorithms}
\label{sec4:methods}

In the section, adaptive BDF2 algorithm and adaptive BDF3 algorithm are given for the purpose of comparing with the experiments in Section \ref{sec5:experiments}.

\begin{longtable}{l}
	\hline
	\textbf{Algorithm 3: Adaptive BDF2 algorithm}\\
	Let $n=2$. Given $\epsilon$, $\hat{\gamma}$, $\check{\gamma}$, $(\u_h^{n-2},\u_h^{n-1},\u_h^{n})$, $\left(p_h^{n-2},p_h^{n-1}\right.$, $\left.p_h^{n}\right)$, and $(\phi_h^{n-2},\phi_h^{n-1},\phi_h^{n})$, compute\\ $(\u_h^{n+1},p_h^{n+1},\phi_h^{n+1})$ by solving\\
	\\
	\qquad\quad$\frac{1}{k_{n+1}}\left(\frac{1+2\tau_n}{1+\tau_n}\u_{h}^{n+1}-\left(1+\tau_n\right)\u_h^n +\frac{\tau_{n}^2}{1+\tau_n}\u_{h}^{n-1},\vv_h\right) +a_f\left(\u_h^{n+1},\vv_h\right)+b\left(\vv_h,p_h^{n+1}\right)$\\ \qquad$=\left(F_1^{n+1},\vv_h\right)-g\int_{\Gamma}\left((1+\tau_n)\phi_h^{n}-\tau_{n}\phi_h^{n-1}\right)\vv_h\cdot\n_f,$\\
	\qquad$b(\u_h^{n+1},q_h)=0,$\\
	\qquad\quad$\frac{gS}{k_{n+1}}\left(\frac{1+2\tau_n}{1+\tau_n}\phi_{h}^{n+1}-\left(1+\tau_n\right)\phi_h^n +\frac{\tau_{n}^2}{1+\tau_n}\phi_{h}^{n-1},\vv_h\right) +a_p\left(\phi_h^{n+1},\varphi_h\right)$\\
	\qquad$=g\left(F_2^{n+1},\varphi_h\right)+g\int_{\Gamma}\varphi_h\left((1+\tau_n)\u_h^{n}-\tau_{n}\u_h^{n-1}\right)\cdot\n_f.$\\
	\\
	\qquad Compute the error estimators:\\
	\qquad\qquad\qquad\qquad\qquad\qquad\qquad$Est_{\u}\leftarrow\eta^3 \varrho^3 \hat{\u}_h^{n+1},$ \quad $Est_{\phi}\leftarrow\eta^3 \varrho^3 \hat{\phi}_h^{n+1}$,\\
	\hline
	\textbf{Algorithm 4: Adaptive BDF3 algorithm}\\
	The fully discrete approximation of adaptive BDF3 algorithm is: Give $\epsilon$, $\hat{\gamma}$, $\check{\gamma}$, $\left(\u_h^0,p_h^0,\phi_h^0\right)$,\\ $\left(\u_h^1,p_h^1,\phi_h^1\right)$, and $\left(\u_h^2,p_h^2,\phi_h^2\right)$, find $\left(\u_h^{n+1},p_h^{n+1},\phi_h^{n+1}\right)\in \left(\mathbf{H}_{fh},Q_h,H_{ph}\right)$, with $n=2,\ldots,N-1$,\\ such that for any $\vv_h\in \mathbf{H}_{fh}$, $q_h\in Q_h$ and $\varphi_h\in H_{ph}$\\
	\\
	\qquad$\left(\frac{\mathcal{A}(\u_h^{n+1})}{k_{n+1}},\vv_h\right) +a_f\left(\u_h^{n+1},\vv_h\right)+b\left(\vv_h,p_h^{n+1}\right) =\left(F_1^{n+1},\vv_h\right)-g\int_{\Gamma}\phi_{h,\sigma}^n\vv_h\cdot\n_f,$\\
	\qquad$b(\u_h^{n+1},q_h)=0,$\\
	\qquad$gS\left(\frac{\mathcal{A}(\phi_h^{n+1})}{k_{n+1}},\varphi_h\right) +a_p\left(\phi_h^{n+1},\varphi_h\right)
	=g\left(F_2^{n+1},\varphi_h\right)+g\int_{\Gamma}\varphi_h\u_{h,\sigma}^n\cdot\n_f.$\\
	\\
	\qquad Compute the error estimators:\\
	\qquad\qquad\qquad\qquad\qquad\qquad\qquad$Est_{\u}\leftarrow\eta^4 \varrho^4 \hat{\u}_h^{n+1},$ \quad $Est_{\phi}\leftarrow\eta^4 \varrho^4 \hat{\phi}_h^{n+1}$,\\
	\hline
	\textbf{Remark:} The adaptive change process of Algorithm 3 and 4 follow the prescribed format.\\
	If $\max \left\{\left|Est_{\u}\right|,\left|Est_{\phi}\right|\right\} <\frac{\epsilon}{4},$ \\
	\qquad\qquad\qquad\qquad$\vartheta_{n+1}=\min \left\{2,\left(\frac{\epsilon}{\left|Est_{\u}\right|}\right)^{\frac{1}{3}}, \left(\frac{\epsilon}{\left|E s t_{\phi}\right|}\right)^{\frac{1}{3}}\right\},$ \qquad$k_{n+1}=\hat{\gamma} \cdot \vartheta_{n+1} \cdot k_n ,$ \\
	if $\frac{\epsilon}{4} \leq \min \left\{\left|E s t_{\u}\right|,\left|E s t_\phi\right|\right\}\leq\epsilon,$ \\
	\qquad\qquad\qquad\qquad$ \vartheta_{n+1}=\min \left\{1,\left(\frac{\epsilon}{\left|E s t_{\u}\right|}\right)^{\frac{1}{3}},\left(\frac{\epsilon}{\left|E s t_\phi\right|}\right)^{\frac{1}{3}}\right\},$ \qquad$k_{n+1}=\hat{\gamma} \cdot \vartheta_{n+1} \cdot k_n .$\\
	If none satisfy the tolerance, set\\
	\qquad\qquad\qquad\qquad\qquad\qquad\qquad\qquad\qquad\qquad $k_{n}=k_n/\hat{\gamma}\cdot\check{\gamma},$\\
	and recompute the above steps.\\
	\hline
\end{longtable}

\section{Numerical tests}
\label{sec5:experiments}

In this section, numerical tests are carried out to further validate the results of theoretical analysis. The first experiment aims to demonstrate that BDF2-TF algorithm can effectively simulate fluid flow under different true solutions and arbitrary time step sequences. Subsequently, the second experiment demonstrates that the convergence order of BDF2-TF is third-order, aligning with the theoretical findings. Additionally, it also shows the high efficiency of time filter method. a simplified model is employed to simulate the process of shale oil extraction from reservoirs, further demonstrating the algorithm's practical applicability in the last experiment.

For the 2D experiments, we utilize continuous piecewise cubic finite element (P3) for fluid velocity $\u$ and hydraulic head $\phi$, whereas continuous piecewise quadratic functions (P2) represent pressure $p$. In 3D experiments, we employ continuous piecewise quadratic finite element (P23d) and continuous piecewise linear finite element (P13d) to represent the Stokes and Darcy region, respectively. The codes are implemented using the FreeFEM++ software package \cite{41}.

\begin{re} In the table, we define global error as
	$Err_{\Phi}=\left(\sum\limits_{i=3}^Nk_i\frac{\left\|\Phi(t_i)-\Phi_{h}^i\right\|_f^2}{\left\|\Phi(t_i)\right\|_f^2}\right)^{\frac{1}{2}}$, and $\rho_{\Phi},$ denote the convergence order of $\Phi$.
\end{re}

\subsection{3D convection in a cubual cavity}

To demonstrate the effectiveness of the proposed algorithm, we employed benchmark tests as suggested in \cite{42}. Specifically, we present a study on 3D natural convections within a cubical cavity. This section will employ two distinct true solutions to simulate velocity streamlines and showcase the computational performance of various algorithms in three dimensions.

All physical parameters, including density $\rho$, gravity $g$, kinematic viscosity $\nu$, hydraulic conductivity $\textbf{K}$, source term coefficient $S$ and thermal expansion coefficient $\alpha$, are set to 1, while the initial conditions, boundary conditions, and source terms are derived from the exact solution. For this experiment, we consider the BDF2-TF algorithm for the Stokes-Darcy model. In order to observe the impact of variable time-stepping on the results, we set the grid size to $h=\frac{1}{7}$. We apply the BDF2-TF algorithm to this test problem for 40 time steps and refer to the
time step size $k_{n_1}$, $k_{n_2}$, and $k_{n_3}$ similar to that in \cite{6}:
\begin{align*}
	&k_{n_1}=0.025+0.0125t_n,\\
	&k_{n_2}=\left\{\begin{array}{cc}
		0.025 & 3\leq n\leq10,\\
		0.025+0.0125\sin(10t_n) & n>10,
	\end{array}\right.\\
	&k_{n_3}=0.025-0.0125t_n.
\end{align*}

\subsubsection{Test 1}
\label{sec6.11}

Assuming the computational domain is $\Omega_f=[0,1]\times [1,2]\times [0,1]$ and $\Omega_p=[0,1]^3$, with the prescribed value of $\u$ at the interface $\Gamma$ being set to $1.0$. Both the initial value and the external force term are initialized as $0$. We consider three cases: BDF2-TF with variable time steps $k_{n_1}$, $k_{n_2}$ and $k_{n_3}$.

\begin{figure}[hbpt!]
	\centering
	\subfloat[$k_{n_1}$]
	{
		\begin{minipage}[hbpt!]{3.2cm}
			\centering
			\includegraphics[width=3.2cm,height=3.2cm]{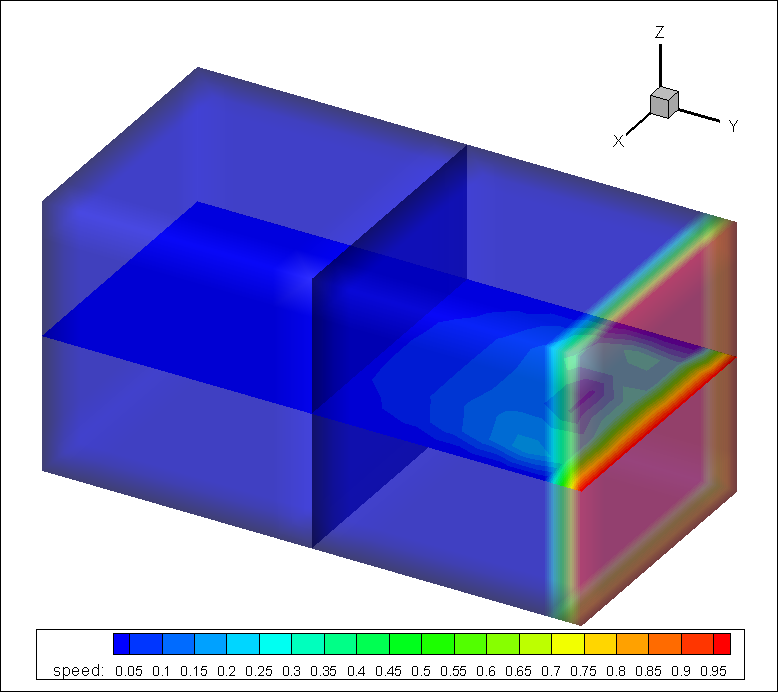}
		\end{minipage}
	}
	\subfloat[$k_{n_2}$]
	{
		\begin{minipage}[hbpt!]{3.2cm}
			\centering
			\includegraphics[width=3.2cm,height=3.2cm]{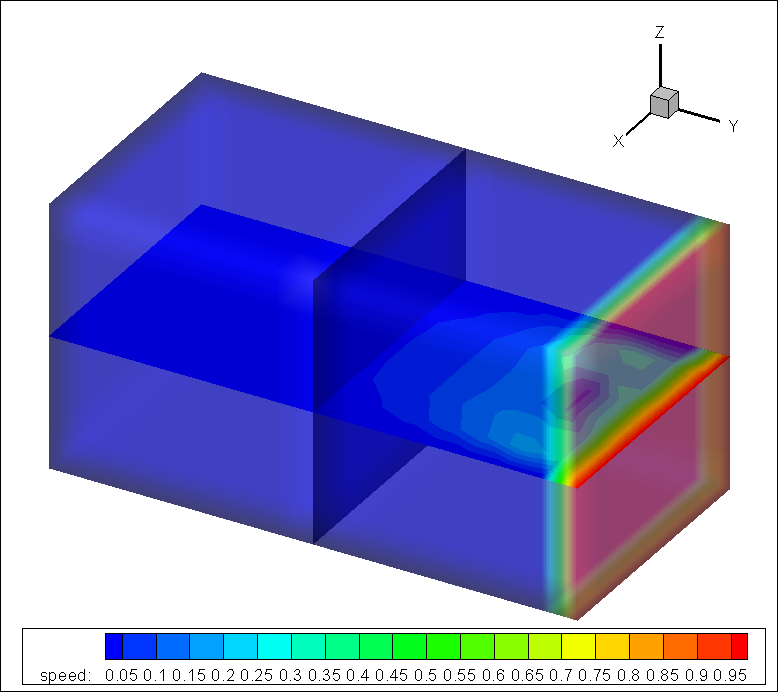}
		\end{minipage}
	}
	\subfloat[$k_{n_3}$]
	{
		\begin{minipage}[hbpt!]{3.2cm}
			\centering
			\includegraphics[width=3.2cm,height=3.2cm]{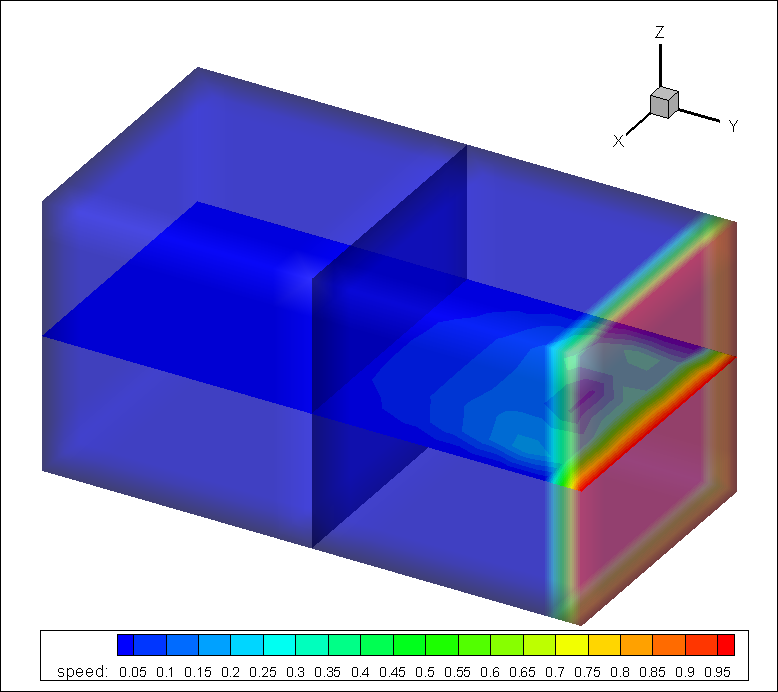}
		\end{minipage}
	}
	\caption{\label{fig1} \small The cubical cavity of BDF2-TF algorithm with different variable time step functions.}
	\centering
	\subfloat[$k_{n_1}$]
	{
		\begin{minipage}[hbpt!]{3.2cm}
			\centering
			\includegraphics[width=3.2cm,height=3.2cm]{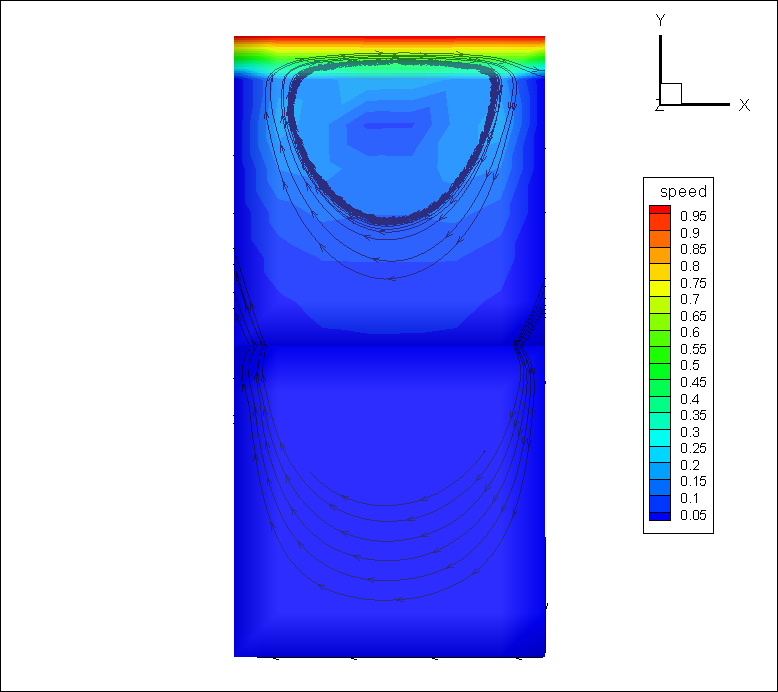}
		\end{minipage}
	}
	\subfloat[$k_{n_2}$]
	{
		\begin{minipage}[hbpt!]{3.2cm}
			\centering
			\includegraphics[width=3.2cm,height=3.2cm]{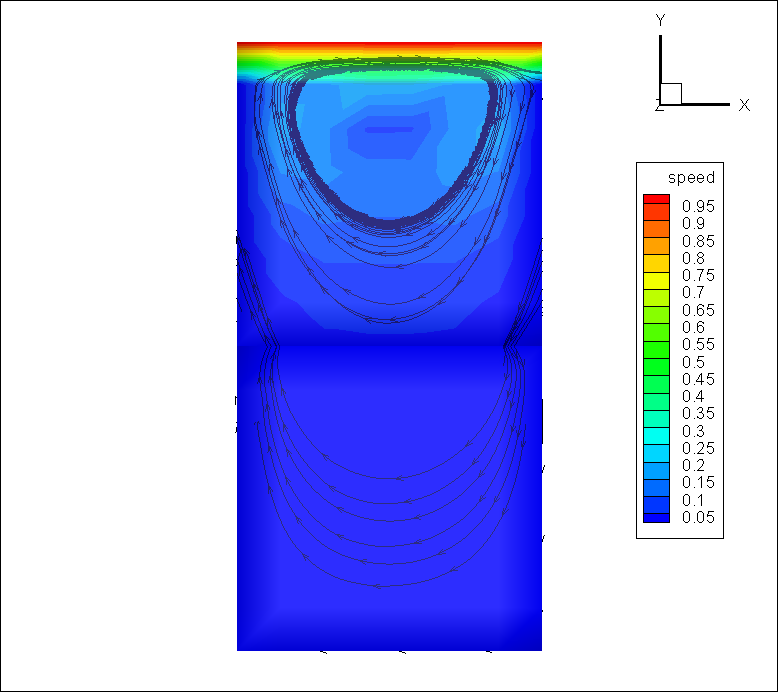}
		\end{minipage}
	}
	\subfloat[$k_{n_3}$]
	{
		\begin{minipage}[hbpt!]{3.2cm}
			\centering
			\includegraphics[width=3.2cm,height=3.2cm]{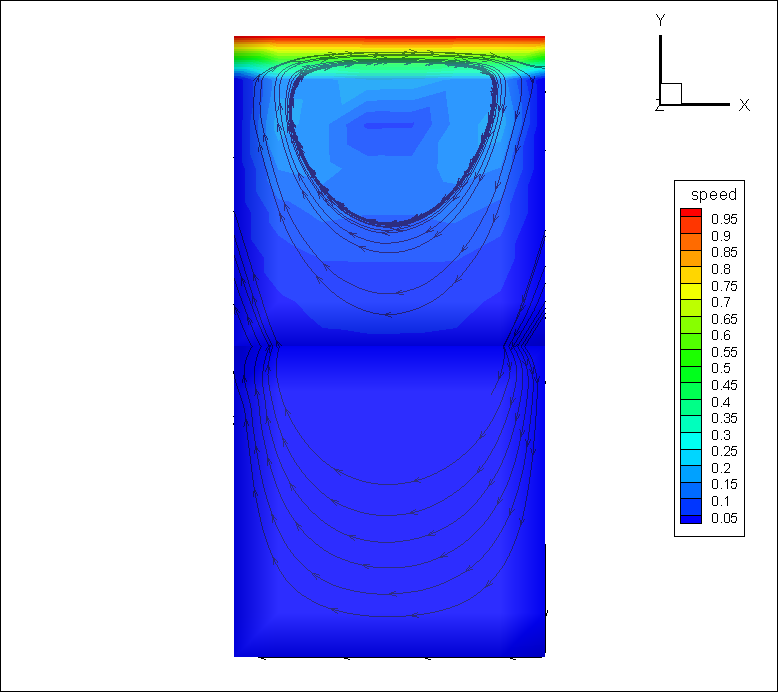}
		\end{minipage}
	}
	\caption{\label{fig2} \small The cross-section view of BDF2-TF algorithm with different variable time step functions at $z=0.5$.}
\end{figure}

From the cubic cavity view depicted and the cross-sectional perspective at $z=0.5$ in Figures \ref{fig1}-\ref{fig2}, it is evident that third-order BDF2-TF algorithm can effectively simulate fluid flow. At the same time, whether the time step increases or decreases, the stability of the algorithm can be reflected.

\begin{figure}[hbpt!]
	\centering
	\subfloat[$k_{n_1}$]
	{
		\begin{minipage}[hbpt!]{3cm}
			\centering
			\includegraphics[width=3cm,height=3cm]{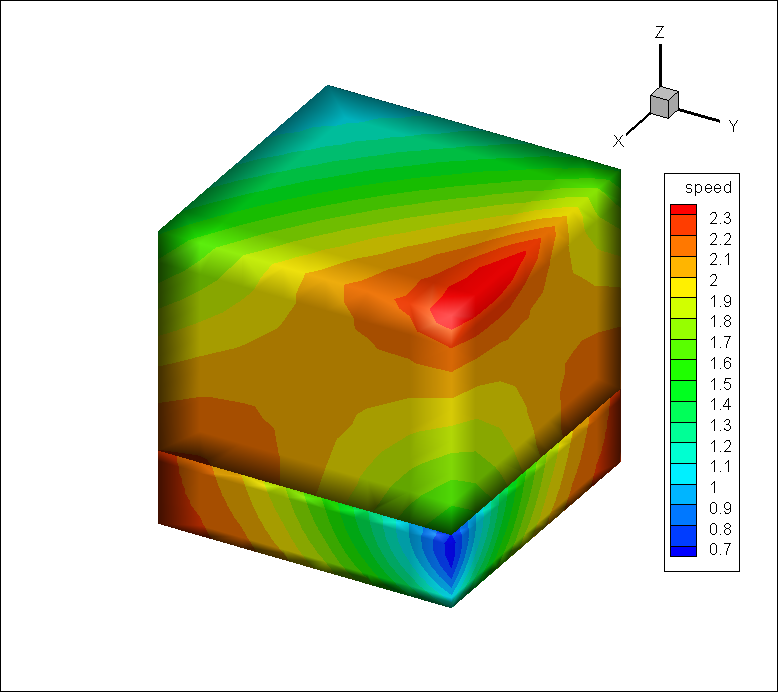}
		\end{minipage}
	}
	\subfloat[$k_{n_2}$]
	{
		\begin{minipage}[hbpt!]{3cm}
			\centering
			\includegraphics[width=3cm,height=3cm]{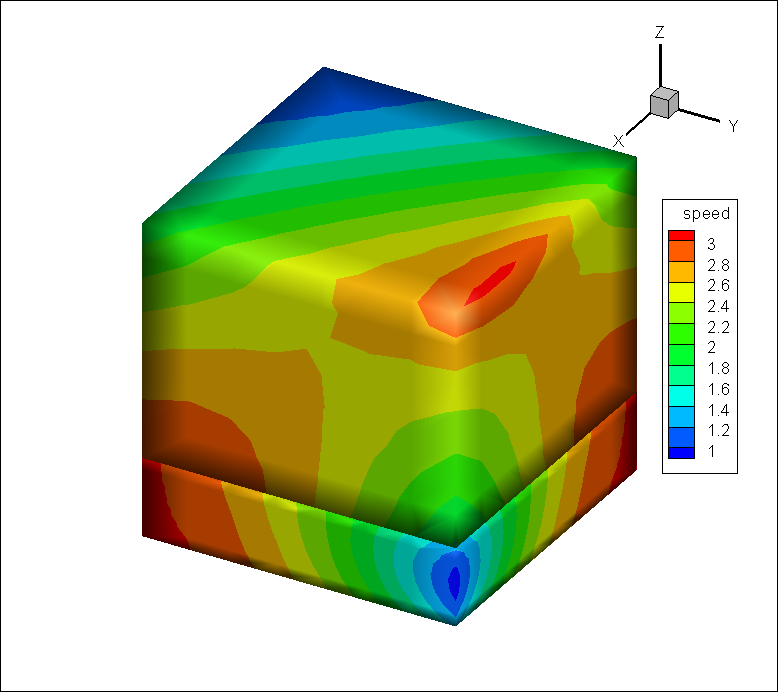}
		\end{minipage}
	}
	\subfloat[$k_{n_3}$]
	{
		\begin{minipage}[hbpt!]{3cm}
			\centering
			\includegraphics[width=3cm,height=3cm]{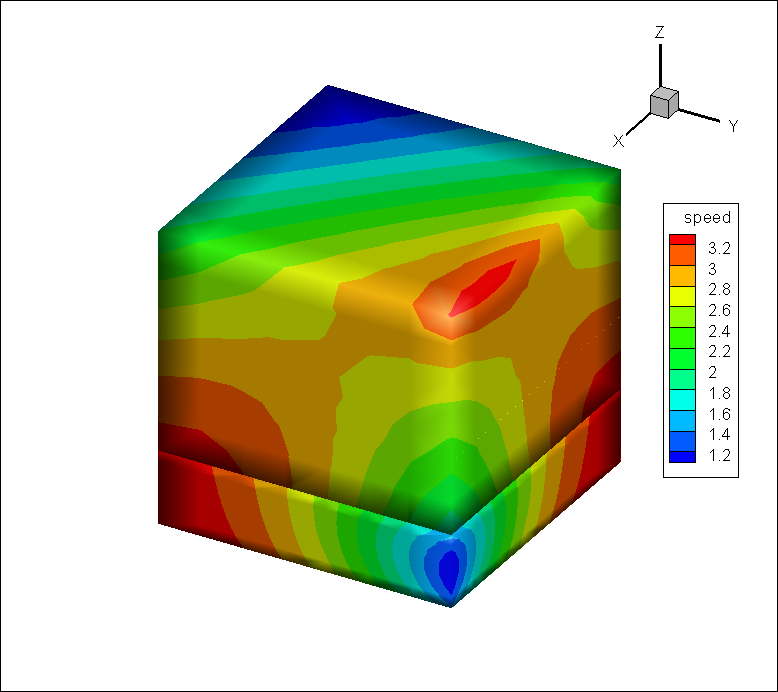}
		\end{minipage}
	}
	\caption{\label{fig3} \small The cubical cavity of BDF2-TF algorithm with different variable time step functions.}
	\centering
	\subfloat[$k_{n_1}$]
	{
		\begin{minipage}[hbpt!]{3cm}
			\centering
			\includegraphics[width=3cm,height=3cm]{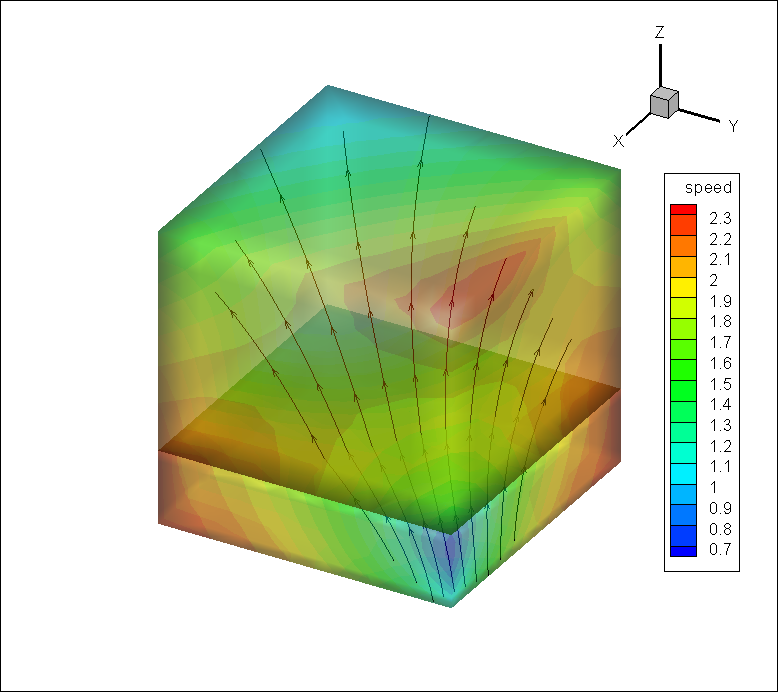}
		\end{minipage}
	}
	\subfloat[$k_{n_2}$]
	{
		\begin{minipage}[hbpt!]{3cm}
			\centering
			\includegraphics[width=3cm,height=3cm]{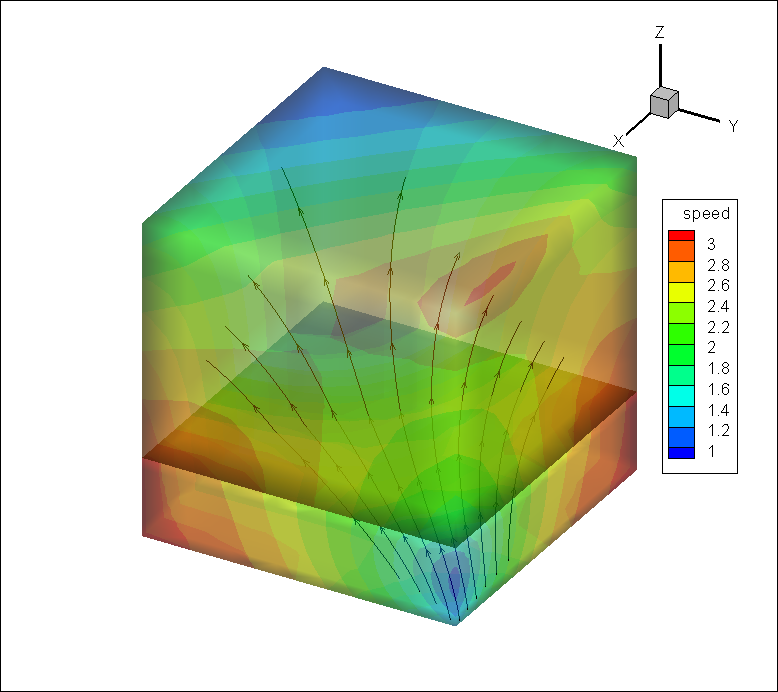}
		\end{minipage}
	}
	\subfloat[$k_{n_3}$]
	{
		\begin{minipage}[hbpt!]{3cm}
			\centering
			\includegraphics[width=3cm,height=3cm]{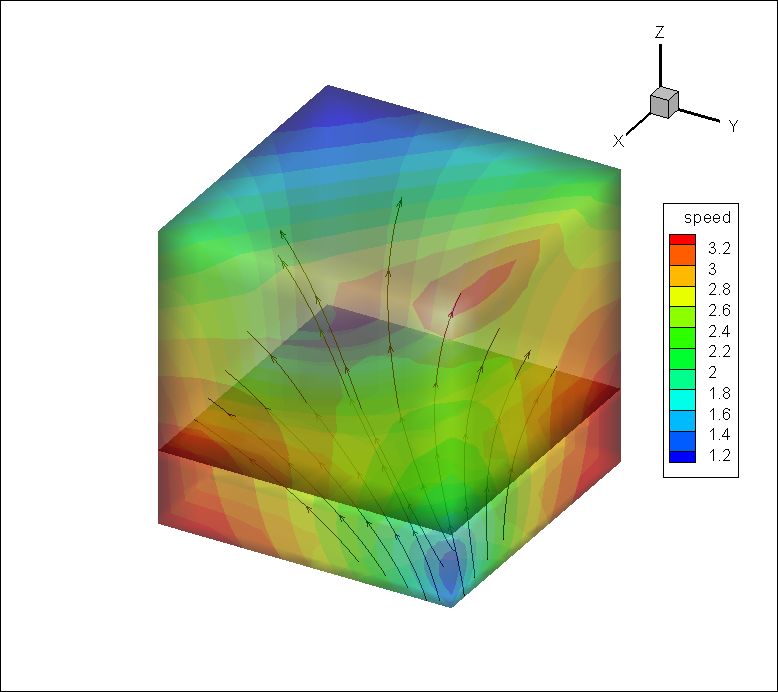}
		\end{minipage}
	}
	\caption{\label{fig4} \small The flow line and velocity size of BDF2-TF algorithm with different variable time step functions.}
	\centering
	\subfloat[$k_{n_1}$]
	{
		\begin{minipage}[hbpt!]{3cm}
			\centering
			\includegraphics[width=3cm,height=3cm]{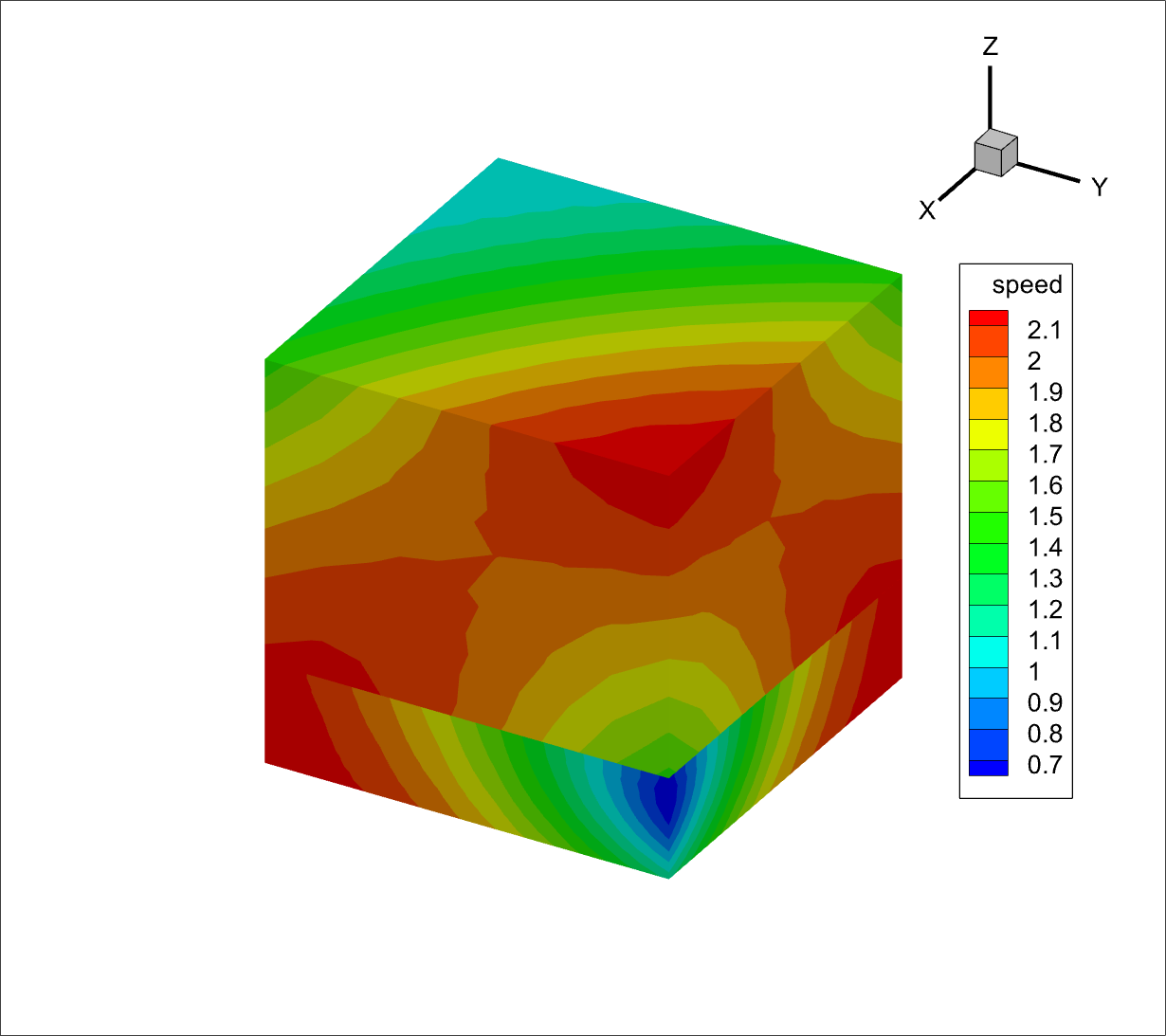}
		\end{minipage}
	}
	\subfloat[$k_{n_2}$]
	{
		\begin{minipage}[hbpt!]{3cm}
			\centering
			\includegraphics[width=3cm,height=3cm]{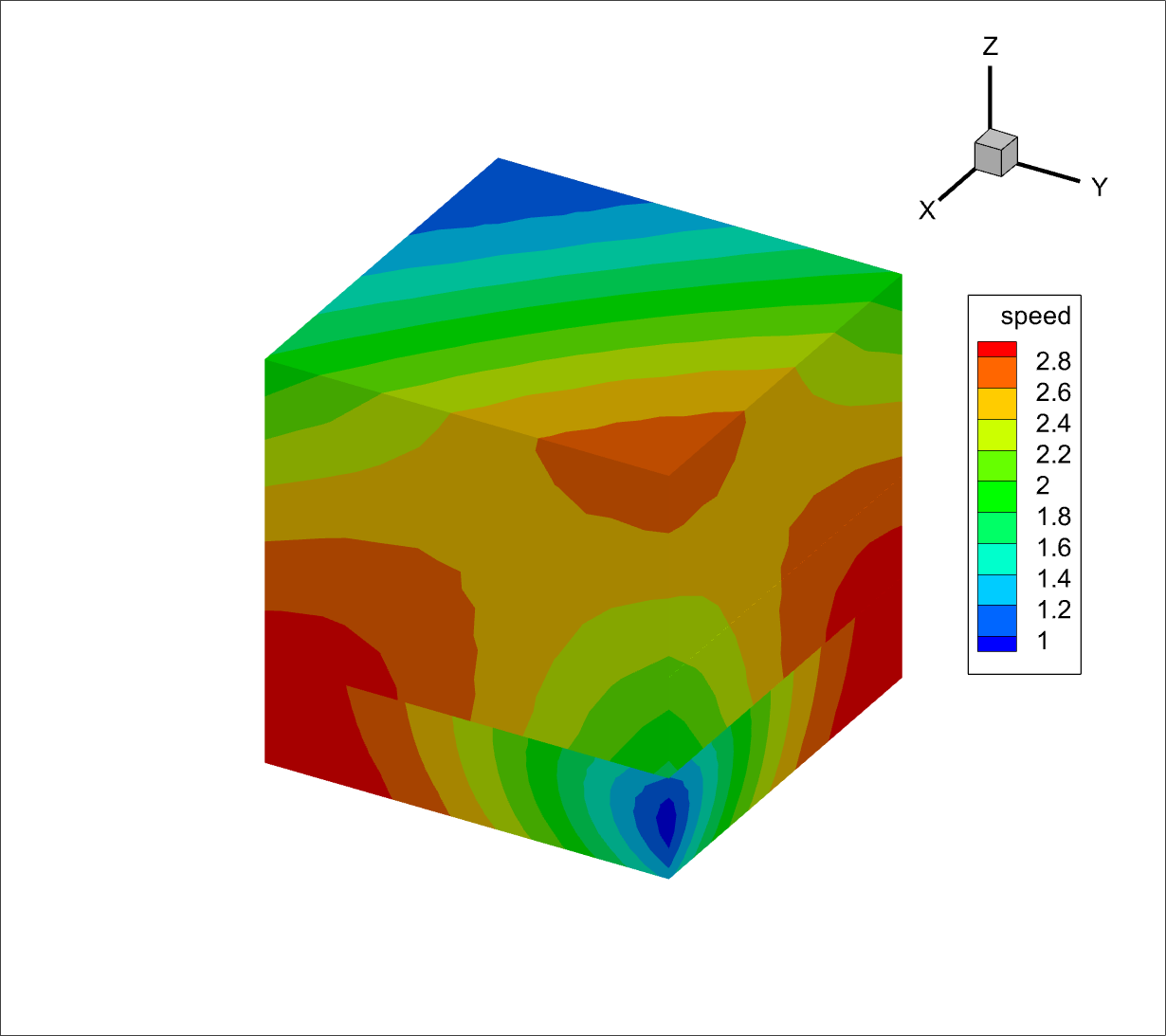}
		\end{minipage}
	}
	\subfloat[$k_{n_3}$]
	{
		\begin{minipage}[hbpt!]{3cm}
			\centering
			\includegraphics[width=3cm,height=3cm]{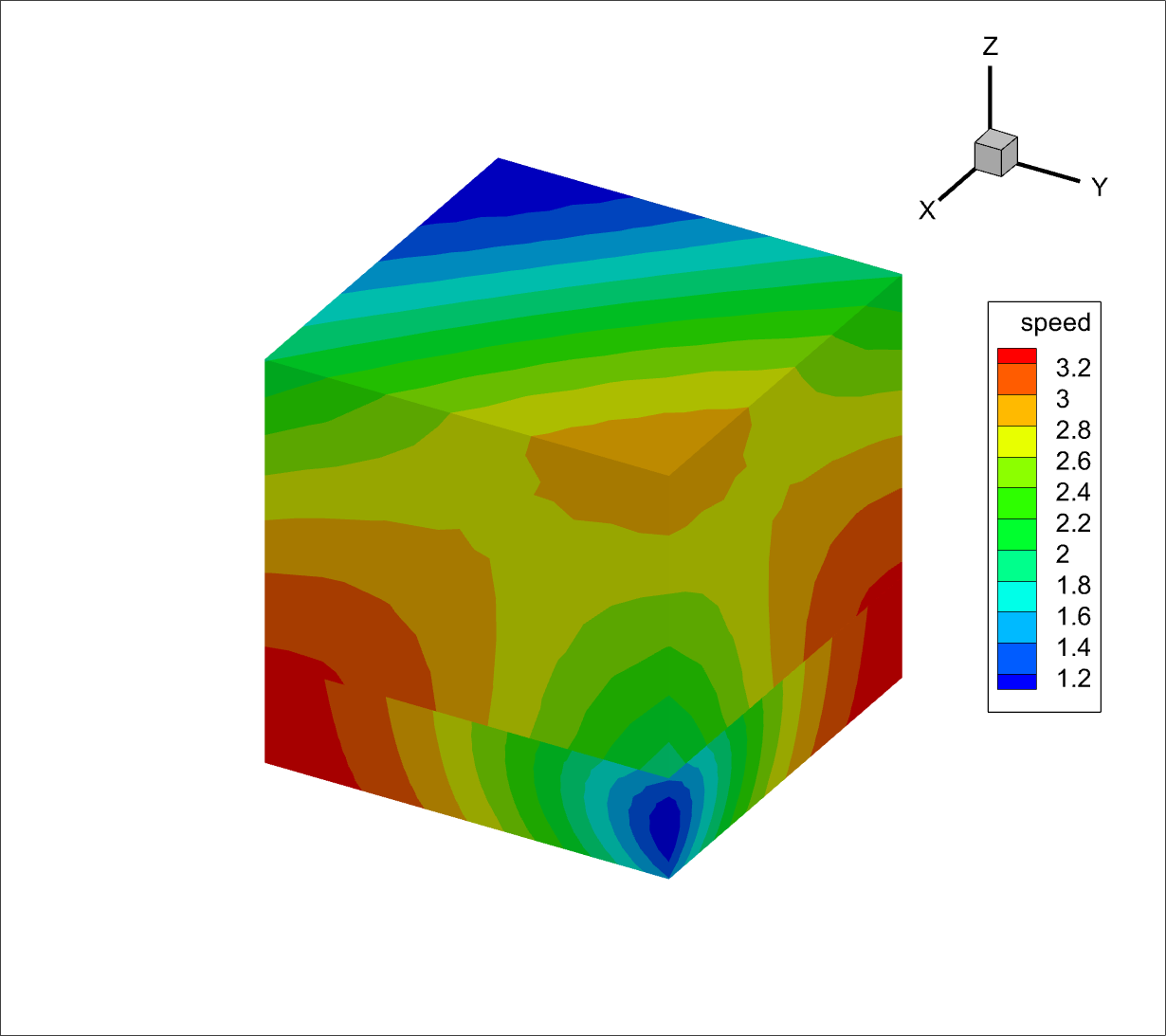}
		\end{minipage}
	}
	\caption{\label{fig5} \small The exact solution is utilized for the computation of the cubical cavity  using the BDF2-TF algorithm with different variable time step functions.}
	\centering
	\subfloat[$k_{n_1}$]
	{
		\begin{minipage}[hbpt!]{3cm}
			\centering
			\includegraphics[width=3cm,height=3cm]{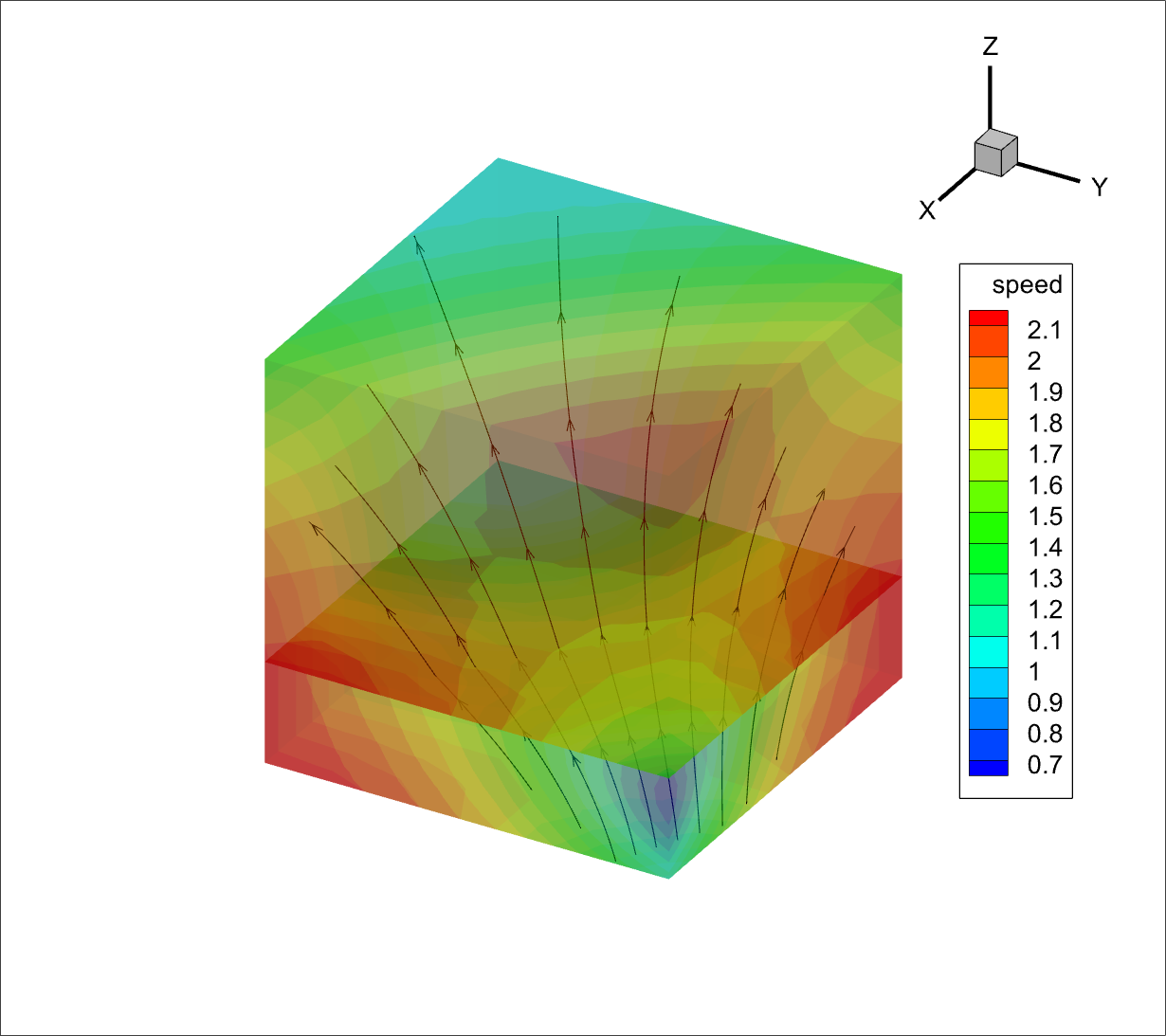}
		\end{minipage}
	}
	\subfloat[$k_{n_2}$]
	{
		\begin{minipage}[hbpt!]{3cm}
			\centering
			\includegraphics[width=3cm,height=3cm]{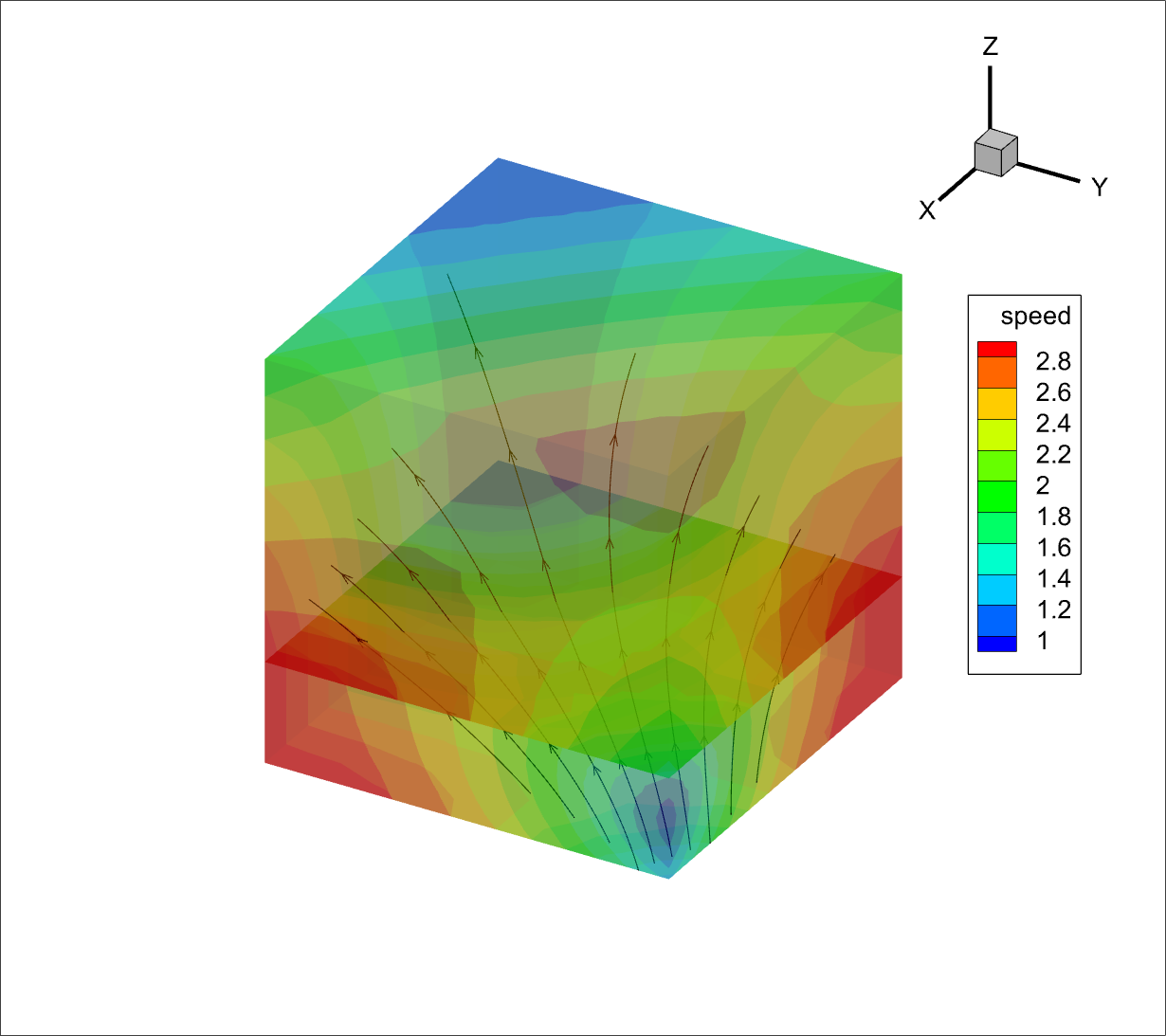}
		\end{minipage}
	}
	\subfloat[$k_{n_3}$]
	{
		\begin{minipage}[hbpt!]{3cm}
			\centering
			\includegraphics[width=3cm,height=3cm]{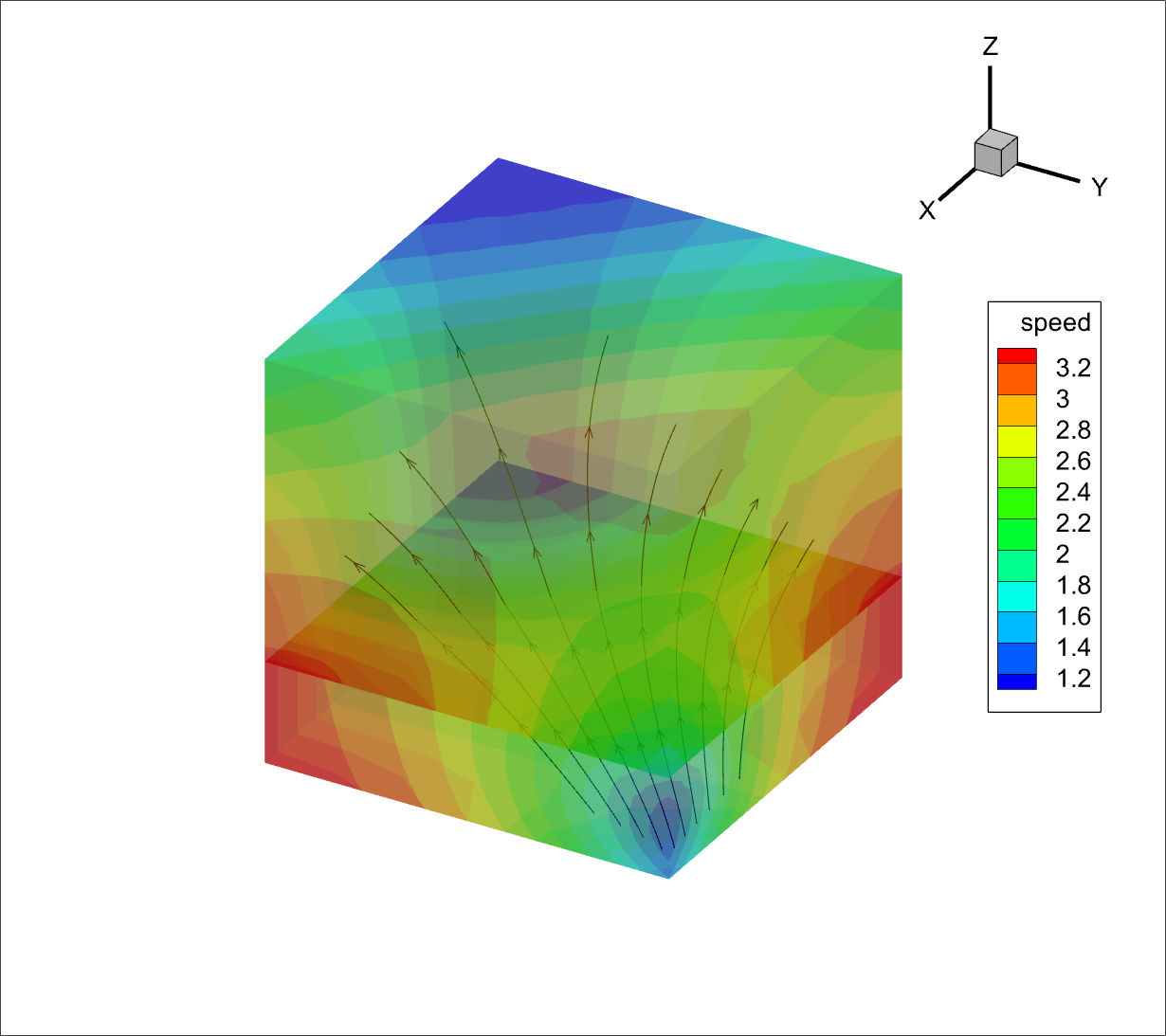}
		\end{minipage}
	}
	\caption{\label{fig6} \small The exact solution is utilized for the computation of the cross-section view using the BDF2-TF algorithm with different variable time step functions at $z=0.5$.}
\end{figure}

\subsubsection{Test 2}
\label{sec6.12}

Let the computational domain $\Omega$ be composed of two regions: the fluid flow region $\Omega_f=[0,1]\times [0,0.25]\times [0,1]$ and the porous media flow region $\Omega_p=[0,1]\times [0.25,0.75]\times [0,1]$, with an interface denoted as $\Gamma=[0,1]\times 0.25 \times [0,1]$. The exact solution is defined as follows:
\begin{align}\label{3d}
	&\u=\left[\begin{array}{c}
		\left(2x\sin(xy)+y(x^2+y^2-8)\cos(xy)\right)\exp(-t)\\
		\left(2y\sin(xy)+x(x^2+y^2-8)\cos(xy)\right)\exp(-t)\\
		1+\left((x^2+y^2)(x^2+y^2-8)\sin(xy)-4\sin(xy)-8xy\cos(xy)\right)z\exp(-t)
	\end{array}\right],\\
	&p=\left(-16xy\cos(xy)+(x^2+y^2+z^2-8)(2x^2+2y^2+2z^2-1)\sin(xy)-8\sin(xy)\right)e^{-t},\\
	&\phi=-z+(-x^2-y^2+8)\exp(-t)\sin(xy)\cos(z).\label{3d3}
\end{align}

We examine the performance of BDF2-TF with variable time-stepping $k_{n_1}$, $k_{n_2}$ and $k_{n_3}$. We successfully created a three-dimensional true solution for the Stokes-Darcy model. It can be seen from Figures \ref{fig3}-\ref{fig4} that BDF2-TF algorithm can effectively simulate fluid flow under any time step sequence. Therefore, the BDF2-TF algorithm we developed has excellent computational performance in three-dimensional case.

Figures \ref{fig5}-\ref{fig6} illustrate the actual variation of the exact solution for different variable time step functions, namely $k_{n_1}$, $k_{n_2}$, and $k_{n_3}$. The numerical solution and the exact solution exhibit a high degree of consistency with the three-dimensional cubical cavity diagram after 40 steps of calculation, thereby demonstrating the superior accuracy and reliability of the adopted BDF2-TF algorithm. Moreover, this algorithm effectively captures the flow characteristics of three-dimensional square cavity flows, thus validating the feasibility and accuracy of our numerical method in solving fluid mechanics problems.

\subsection{Convergence test}

In this part, we will test the convergence order of adaptive time step and variable time step for BDF2, BDF2-TF and BDF3 algorithms in 2D and 3D domains, respectively. All physical parameters, including density $\rho$, gravity $g$, kinematic viscosity $\nu$, hydraulic conductivity $\textbf{K}$, source term coefficient $S$ and thermal expansion coefficient $\alpha$, are set to 1. The final time is fixed at $T=1.0$, while the initial conditions, boundary conditions, and source terms are derived from the exact solution.

\subsubsection{Test the convergence rate in 2D}
\label{sec6.22D}

Set the computational domain $\Omega_f=(0,1)\times (1,2)$, $\Omega_p=(0,1)\times (0,1)$, with the interface $\Gamma=(0,1)\times {1}$. The exact solution is given by:
\begin{align*}
	&\u(x,y,t)=\left(\left[x^2(y-1)^2+y\right]\cos(t),\left[-\frac{2}{3}x(y-1)^3+2-\pi\sin(\pi x)\right]\cos(t)\right)^{T},\\
	&p(x,y,t)=\left[2-\pi\sin(\pi x)\right]\sin\left(\frac{\pi y}{2}\right)\cos(t),\\
	&\phi(x,y,t)=\left[2-\pi\sin(\pi x)\right]\left[1-y-\cos(\pi y)\right]\cos(t).
\end{align*}

\begin{table}[hbpt!]
	\caption{The convergence order and CPU time for adaptive BDF2 algorithm in 2D.}
	\label{tab:2}
	\centering\begin{tabular}{cccccc}
		\hline
		$\qquad\overline{\Delta t}\qquad$ & $\quad Err_{\u}\quad$ &$\quad \rho_{\u}\quad$ &  $\quad Err_{\phi}\quad$ &$\quad \rho_{\phi}\quad$ &$\qquad CPU(s)\qquad$ \\
		\noalign{\smallskip}\hline\noalign{\smallskip}
		$\frac{1}{15}$ & 5.31173$\mathrm{e}-5$ & - & 0.000449915& - &711.86\\
		$\frac{1}{32}$ & 1.51286$\mathrm{e}-5$ & 1.6576 & 0.000106635  &  1.9001 & 1499.59\\
		$\frac{1}{65}$ & 2.87312$\mathrm{e}-6$ & 2.3442 &2.44551$\mathrm{e}-5$  &  2.0780 & 3037.93\\
		$\frac{1}{159}$ &   4.52788$\mathrm{e}-7$ & 2.0656 & 3.65718$\mathrm{e}-6$ & 2.1242 & 7204\\
		\hline
	\end{tabular}
\end{table}
\begin{table}[hbpt!]
	\caption{The convergence order and CPU time for adaptive BDF2-TF algorithm in 2D.}
	\label{tab:3}
	\centering\begin{tabular}{cccccc}
		\hline
		$\qquad\overline{\Delta t}\qquad$ & $\quad Err_{\u}\quad$ &$\quad \rho_{\u}\quad$ &  $\quad Err_{\phi}\quad$ &$\quad \rho_{\phi}\quad$ &$\qquad CPU(s)\qquad$ \\
		\noalign{\smallskip}\hline\noalign{\smallskip}
		$\frac{1}{16}$ & 4.29182$\mathrm{e}-5$ & - & 3.55191$\mathrm{e}-5$ & - & 678.113\\
		$\frac{1}{33}$ & 4.57457$\mathrm{e}-6$ & 3.0926 & 3.73058$\mathrm{e}-6$  &  3.1129 &1303.34 \\
		$\frac{1}{71}$ & 4.26536$\mathrm{e}-7$ & 3.0967 & 3.37895$\mathrm{e}-7$  &  3.1345 & 2850.56\\
		$\frac{1}{167}$ &  2.54757$\mathrm{e}-8$ & 3.2947 &  2.09852$\mathrm{e}-8$ & 3.2490 & 7039.2\\
		\hline
	\end{tabular}
\end{table}
\begin{table}[hbpt!]
	\caption{The convergence order and CPU time for adaptive BDF3 algorithm in 2D.}
	\label{tab:4}
	\centering\begin{tabular}{cccccc}
		\hline
		$\qquad\overline{\Delta t}\qquad$ & $\quad Err_{\u}\quad$ &$\quad \rho_{\u}\quad$ &  $\quad Err_{\phi}\quad$ &$\quad \rho_{\phi}\quad$ &$\qquad CPU(s)\qquad$ \\
		\noalign{\smallskip}\hline\noalign{\smallskip}
		$\frac{1}{22}$ &  4.27834$\mathrm{e}-7$ & - &  4.31139$\mathrm{e}-6$ & - & 3036.27\\
		$\frac{1}{47}$ & 5.42544$\mathrm{e}-8$ & 2.7204 & 5.4627$\mathrm{e}-7$ & 2.7215 & 7497.93\\
		$\frac{1}{97}$ & 6.82596$\mathrm{e}-9$ & 2.8610 & 6.86739$\mathrm{e}-8$  &  2.8621 &14615.8 \\
		$\frac{1}{197}$ & 8.77076$\mathrm{e}-10$ & 2.8961 & 8.64487$\mathrm{e}-9$  &  2.9251 & 29553.1\\
		\hline
	\end{tabular}
\end{table}

Simulations were performed at the final time $T$ using BDF2, BDF2-TF, and BDF3 with adaptive time step, the mesh size is set to $h=\frac{1}{120}$. $\overline{\Delta t}$ represents the average time step. Let $\hat{\gamma}=1.0$ and $\check{\gamma}=0.5$, with various tolerances ranging from $10^{-4}$ to $10^{-7}$. The results in Table \ref{tab:2} show that the convergence order of the BDF2-TF algorithm in the two-dimensional program is third order, which is consistent with the theoretical results. By comparison, the time filter method effectively improves the accuracy on the basis of almost no additional computation. In Table \ref{tab:3}-\ref{tab:4}, it can be observed that at the same convergence order, adaptive BDF2-TF reduces the total effort (number of steps taken) required by fewer time steps compared to adaptive BDF3.

\subsubsection{Test the convergence rate in 3D}
\label{sec6.23D}

Consider the model in 3D, let $\Omega=(0,1)\times (0,2)\times (0,1)$ with $\Omega_f=\{(x,y,z)\in \Omega|y\geq1\}$ and $\Omega_p=\{(x,y,z)\in \Omega|y\leq1\}$, and $\Gamma=\{(x,y,x)\in\Omega|y=1\}$. We utilize the exact solution of (\ref{3d})-(\ref{3d3}) in Test \ref{sec6.12}.
\begin{table}[hbpt!]
	\caption{The convergence order and CPU time for variable time step BDF2 algorithm in 3D.}
	\label{tab:5}
	\centering\begin{tabular}{cccccccc}
		\hline
		$\quad\overline{\Delta t}$ & $\quad Err_{\u}$ & $\quad\kappa_{\u}$ &  $\quad Err_p$ & $\quad\kappa_{p}$ & $\quad Err_{\phi}$ & $\quad\kappa_{\phi}$ & $\quad$CPU(s)\\
		\noalign{\smallskip}\hline\noalign{\smallskip}
		$\frac{1}{50}$ &  6.29712$\mathrm{e}-6$ & - &  0.000353835 & - & 1.47993$\mathrm{e}-5$ & - & 539.139 \\
		$\frac{1}{100}$ & 1.57566$\mathrm{e}-6$ & 1.9949 & 8.74589$\mathrm{e}-5$ & 2.0008 & 3.66499$\mathrm{e}-6$ & 2.0003 & 1076.51 \\
		$\frac{1}{200}$ & 4.01952$\mathrm{e}-7$ & 1.9709 & 2.17446$\mathrm{e}-5$ &  2.0079  &  9.14392$\mathrm{e}-7$ & 2.0029 & 2355.4 \\
		$\frac{1}{400}$ & 1.06336$\mathrm{e}-7$ & 1.9184 & 5.42315$\mathrm{e}-6$ &  2.0035 & 2.29601$\mathrm{e}-7$ & 1.9937 & 4209.78 \\
		\hline
	\end{tabular}
\end{table}
\begin{table}[hbpt!]
	\caption{The convergence order and CPU time for variable time step BDF2-TF algorithm in 3D.}
	\label{tab:6}
	\centering\begin{tabular}{cccccccc}
		\hline
		$\quad\overline{\Delta t}$ & $\quad Err_{\u}$ & $\quad\kappa_{\u}$ &  $\quad Err_p$ & $\quad\kappa_{p}$ & $\quad Err_{\phi}$ & $\quad\kappa_{\phi}$ & $\quad$CPU(s)\\
		\noalign{\smallskip}\hline\noalign{\smallskip}
		$\frac{1}{50}$ &  6.36928$\mathrm{e}-6$ & - &  1.28441$\mathrm{e}-5$ & - & 1.43813$\mathrm{e}-6$ & - & 515.041 \\
		$\frac{1}{100}$ & 7.37954$\mathrm{e}-7$ & 3.1095 & 1.58327$\mathrm{e}-6$ & 3.0201 & 1.69195$\mathrm{e}-7$ & 3.0874 & 1070.18 \\
		$\frac{1}{200}$ & 6.02095$\mathrm{e}-8$ & 3.6155 & 1.78694$\mathrm{e}-7$ &  3.1473  &  1.57249$\mathrm{e}-8$ & 3.4276 & 2334.43 \\
		$\frac{1}{400}$ & 1.02827$\mathrm{e}-8$ & 2.5498 & 1.99275$\mathrm{e}-8$ &  3.1647 & 1.34739$\mathrm{e}-9$ & 3.5448 & 4301.48 \\
		\hline
	\end{tabular}
\end{table}
\begin{table}[hbpt!]
	\caption{The convergence order and CPU time for variable time step BDF3 algorithm in 3D.}
	\label{tab:7}
	\centering\begin{tabular}{cccccccc}
		\hline
		$\quad\overline{\Delta t}$ & $\quad Err_{\u}$ & $\quad\kappa_{\u}$ &  $\quad Err_p$ & $\quad\kappa_{p}$ & $\quad Err_{\phi}$ & $\quad\kappa_{\phi}$ & $\quad$CPU(s)\\
		\noalign{\smallskip}\hline\noalign{\smallskip}
		$\frac{1}{50}$ &  1.14709$\mathrm{e}-6$ & - &  6.80998$\mathrm{e}-5$ & - & 2.77866$\mathrm{e}-6$ & - & 589.239 \\
		$\frac{1}{100}$ & 1.36846$\mathrm{e}-7$ & 2.8952 & 8.24132$\mathrm{e}-6$ & 2.8746 & 3.36944$\mathrm{e}-7$ & 2.8717 & 1161.499 \\
		$\frac{1}{200}$ & 1.56288$\mathrm{e}-8$ & 2.9590 & 1.01156$\mathrm{e}-6$ &  2.8543  &  4.09574$\mathrm{e}-8$ & 2.8682 & 2578.49 \\
		$\frac{1}{400}$ & 1.65364$\mathrm{e}-9$ & 3.0742 & 1.26272$\mathrm{e}-7$ &  2.8304 & 4.7064$\mathrm{e}-9$ & 2.9499 & 5181.81 \\
		\hline
	\end{tabular}
\end{table}

For the 3D variable time step BDF2, BDF2-TF and BDF3 algorithms, we opt for uniform triangular meshes with a fixed mesh size of $h=\frac{1}{7}$. The time step change is calculated by the sum of the current time step and the time step that changes periodically according to the sinusoidal function. In this particular context, we employ the same algorithm as described in \cite{3} to compute the convergence order with respect concerning the time step $\Delta t$. Therefore, we define
$$\rho_v=\frac{\left\|v_h^{\Delta t}-v_h^{\frac{\Delta t}{2}}\right\|_0}{\left\|v_h^{\frac{\Delta t}{2}}-v_h^{\frac{\Delta t}{4}}\right\|_0}.$$
When considering variables such as velocity $\u$, pressure $p$, and  hydraulic head $\phi$, $\rho_v$ can be approximated to approximately $\frac{(4^\kappa-2^\kappa)}{(2^\kappa-1)}$. Specifically, for $\kappa=2$, the corresponding order of convergence in time is $O(\Delta t^2)$, resulting in $\rho_v$ approximately 4. For $\kappa=3$, the order of convergence increases to $O(\Delta t^3)$ and $\rho_v$ becomes approximately 8.

The test employs variable time step to detect the convergence order of the three-dimensional BDF2, BDF2-TF, and BDF3 algorithms. The results demonstrate that the numerical convergence rates of BDF2, BDF2-TF, and BDF3 for $\u$ and $\phi$ are approximately second order, third order, and third order respectively in the three-dimensional program. By comparing Table \ref{tab:5} and Table \ref{tab:6}, it is evident that employing a  time filter can raise the convergence order from second order to third order without increasing the computation time. Furthermore, by comparing Table \ref{tab:6} and Table \ref{tab:7}, it becomes apparent that for equal convergence order, the computation time of BDF2-TF is lower than of BDF3, thereby highlighting the computational efficiency achieved through utilizing a time filter.
\begin{figure}[hbpt!]
	\centering
	{
		\begin{minipage}[hbpt!]{5cm}
			\centering
			\includegraphics[width=5cm,height=5cm]{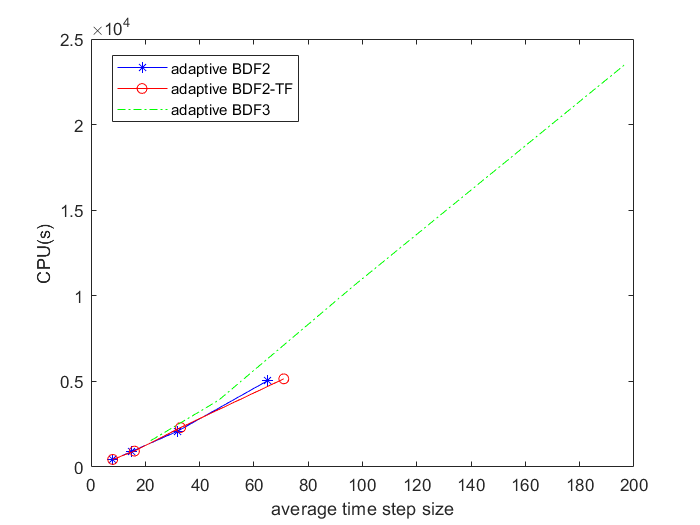}
		\end{minipage}
	}
	{
		\begin{minipage}[hbpt!]{5cm}
			\centering
			\includegraphics[width=5cm,height=5cm]{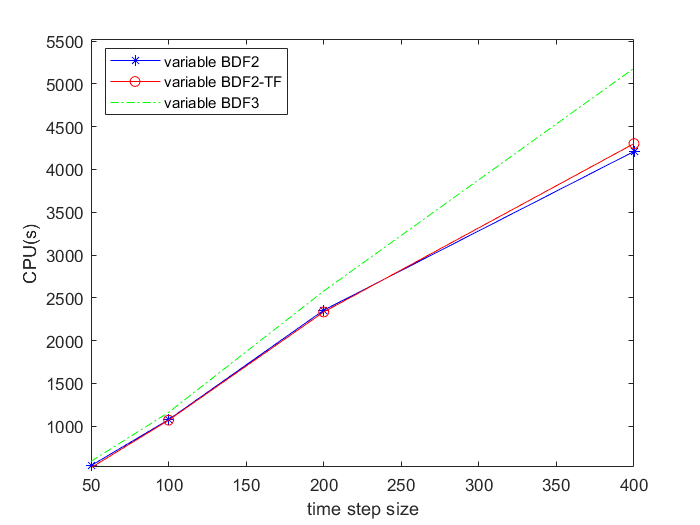}
		\end{minipage}
	}
	\caption{\label{fig14} \small Line graphs of the relationship between CPU calculating time and time step for adaptive and variable time step algorithms.}
\end{figure}

Figure \ref{fig14} describes the relationship between CPU computation time and time step, showing that the time filter method can improve the convergence order without increasing the computation time.

\subsection{Two-face horizontal open-holes attached with a vertical production wellbore and two injection wellbores completion}

Due to the micrometer and nano-scale pore structure in shale reservoirs, traditional hydrodynamic models struggle to accurately describe the flow behavior of these reservoirs. The Stokes-Darcy model provides a detailed depiction of fluid movement in both free flow regions through the application of Stokes equation and porous media via Darcy's law. This model effectively captures complex flow phenomena within the shale microstructure. By utilizing this model, researchers can gain better insights into fluid flow laws in shale reservoirs, optimize production strategies, and ultimately enhance shale oil recovery and economic benefits. The application of this model not only advances reservoir engineering development but also establishes a scientific foundation for efficient energy resource utilization.
\begin{figure}[hbpt!]
	\centering
	{
		\begin{minipage}[hbpt!]{5cm}
			\centering
			\includegraphics[width=5cm,height=5cm]{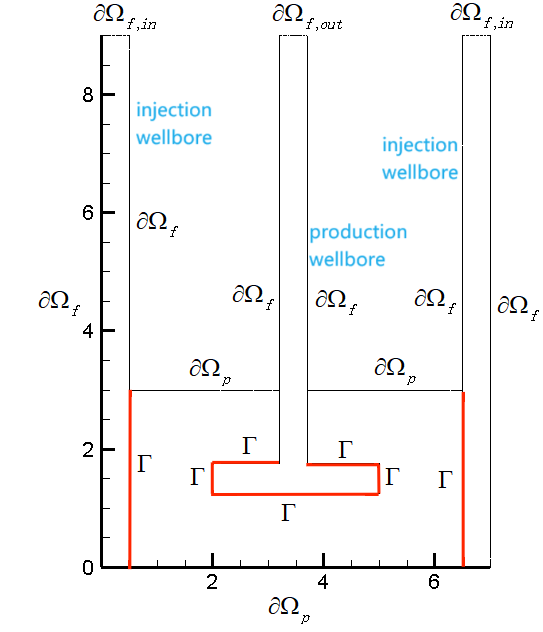}
		\end{minipage}
	}
	{
		\begin{minipage}[hbpt!]{5cm}
			\centering
			\includegraphics[width=5cm,height=5cm]{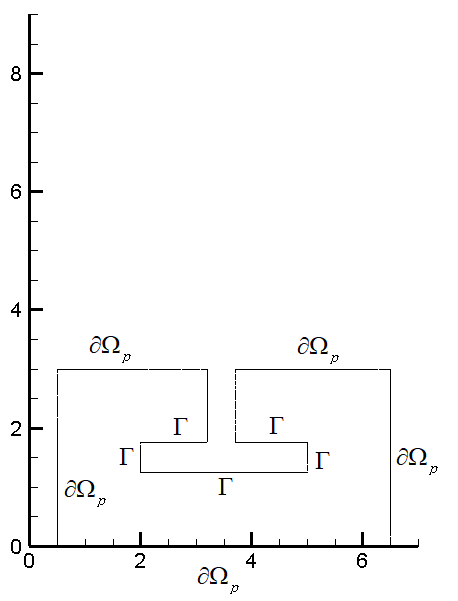}
		\end{minipage}
	}
	\caption{\label{fig15} \small Sketch of the computational domain for the two face horizontal wellbore attached with vertical production well and two injection wellbore completion.}
\end{figure}

The completion of a two-face horizontal open-hole attached with a vertical production wellbore and two injection wellbores represents a simplified conceptual domain for the extraction of hydrocarbons from a porous medium reservoir. The computational domain is depicted in Figure \ref{fig15}. The primary objective of this numerical example is to simulate fluid flow using the Stokes-Darcy model. Inflow boundary conditions are imposed on the top of the two injection wellbores, with $\partial\Omega_{f,in}=-4096(0.25-x)$ applied to the left vertical well and $\partial\Omega_{f,in}=-4096(x-6.75)(7-x)$ applied to the right vertical well. On the contrary, a homogeneous Neumann boundary condition is enforced at the top of the production wellbore by setting $\partial\Omega_{f,in}=(-pI+\nu\nabla\u)\cdot\n_f=0$ to ensure free outflow. The other boundaries $\partial\Omega_f\backslash\Gamma$ of the incompressible flow region are assumed to have a no-slip boundary condition with $\u=0$. Additionally, a constant pressure boundary condition of $\phi=10^4$ is imposed for the porous media flow region. The hydraulic conductivity and kinetic viscosity are given as $\textbf{K}=0.1$ and $\nu=10^{-3}$ respectively. The interface between the incompressible flow region and porous media flow region satisfies three interface conditions (\ref{BJS}). Both initial values and external force terms are initialized as $0$. The mesh size and final time are considered as $h=\frac{1}{40}$ and $T=10.0$. We consider the case: BDF2, BDF2-TF, and BDF3 with variable time steps $k_{n_5}$
\begin{align*}
	&k_{n_5}=\left\{\begin{array}{cc}
		0.01 & 3\leq n\leq10,\\
		0.01+0.005\sin(10t_n) & n>10.
	\end{array}\right.
\end{align*}

\begin{figure}[hbpt!]
	\centering
	\subfloat[BDF2(5517.36s)]
	{
		\begin{minipage}[hbpt!]{4.5cm}
			\centering
			\includegraphics[width=4.5cm,height=4.5cm]{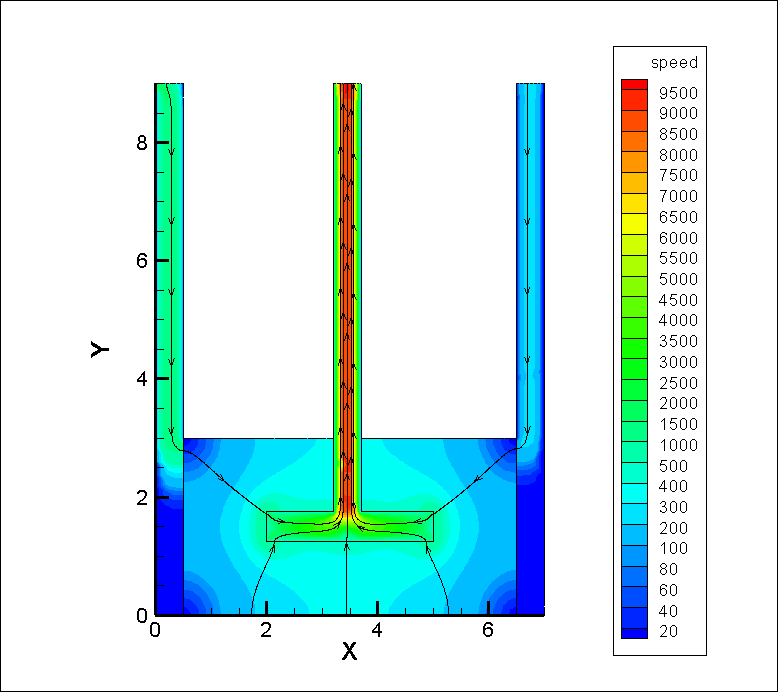}
		\end{minipage}
	}
	\subfloat[BDF2-TF(5329.19s)]
	{
		\begin{minipage}[hbpt!]{4.5cm}
			\centering
			\includegraphics[width=4.5cm,height=4.5cm]{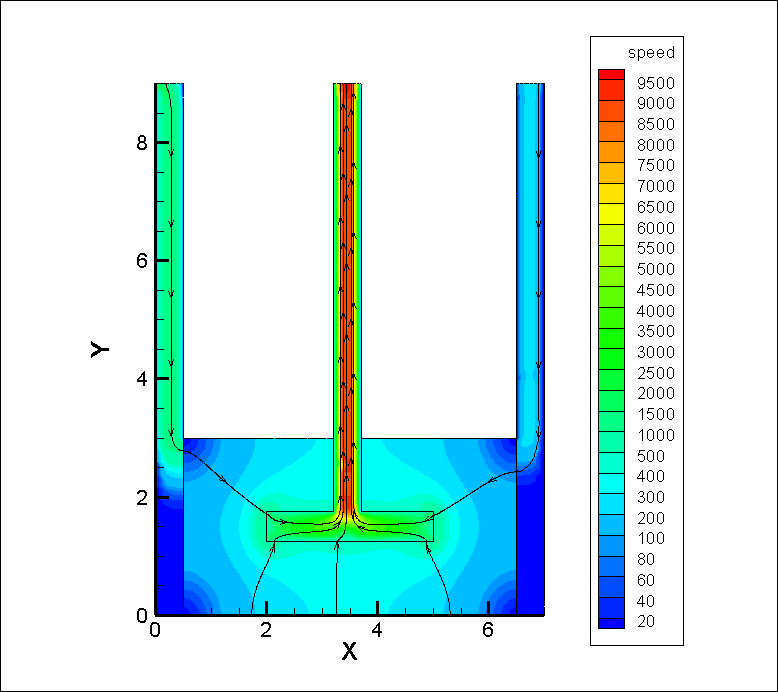}
		\end{minipage}
	}
	\subfloat[BDF3(7351.34s)]
	{
		\begin{minipage}[hbpt!]{4.5cm}
			\centering
			\includegraphics[width=4.5cm,height=4.5cm]{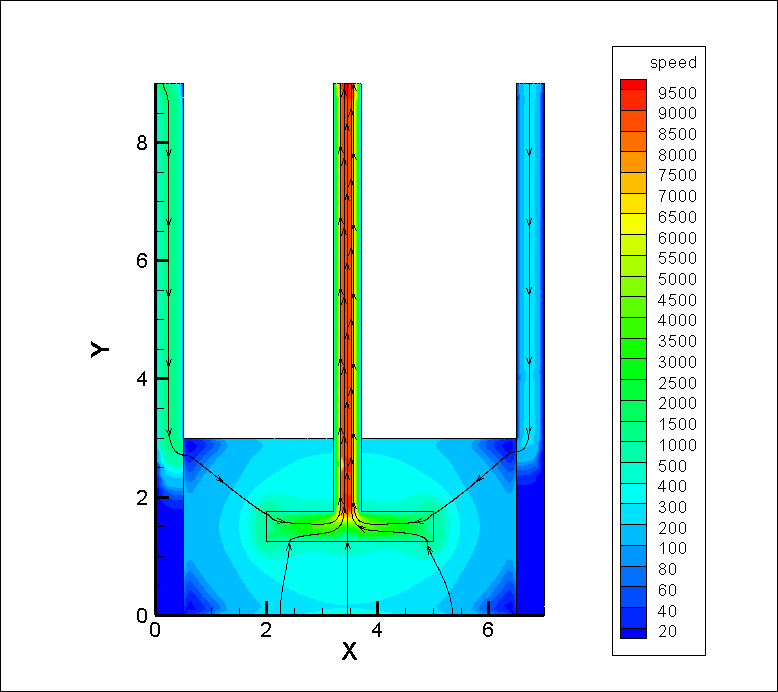}
		\end{minipage}
	}
	\caption{\label{fig16} \small Flow speed and streamlines in the Stokes-Darcy domain with variable time-stepping $k_{n_5}$.}
\end{figure}

The Figure \ref{fig16} illustrates the outcomes of simulating shale oil flows using variable time-stepping algorithms (BDF2, BDF2-TF, and BDF3) within the framework of the Stokes-Darcy model. By comparing the velocity field and streamline distribution resulting from the algorithms, it is evident that they exhibit slight variations in their impact on the velocity field during simulation; however, overall results remain similar, indicating their effective handling of the coupled model. This diagram serves as a valuable reference for researchers to select an appropriate time step strategy that balances calculation accuracy with cost. The computation time of three algorithms (BDF2, BDF2-TF, and BDF3) in simulating shale oil production is compared. The results demonstrate that the BDF2-TF algorithm exhibits the shortest execution time, indicating its superior computational efficiency.

\bibliographystyle{plain}
\bibliography{BDF2_TF}

\begin{thebibliography}{10}

\bibitem{7}
T.~Arbogast and D.~S. Brunson.
\newblock A computational method for approximating a darcy-stokes system
  governing a vuggy porous medium.
\newblock {\em Computational geosciences}, 11:207--218, 2007.

\bibitem{21}
W.~B. Chen, X.~M. Wang, Y.~Yan, and Z.~Y. Zhang.
\newblock A second order bdf numerical scheme with variable steps for the
  cahn-hilliard equation.
\newblock {\em SIAM Journal on Numerical Analysis}, 57(1):495--525, 2019.

\bibitem{8}
P.~Chidyagwai and B.~Rivi$\grave{e}$re.
\newblock On the solution of the coupled navier-stokes and darcy equations.
\newblock {\em Computer Methods in Applied Mechanics and Engineering},
  198(47):3806--3820, 2009.

\bibitem{43}
G.~G. Dahlquist.
\newblock G-stability is equivalent to a-stability.
\newblock {\em BIT Numerical Mathematics}, 18:384--401, 1978.

\bibitem{25}
V.~DeCaria, A.~Guzel, W.~J. Layton, and Y.~Li.
\newblock A variable stepsize, variable order family of low complexity.
\newblock {\em SIAM Journal on Scientific Computing}, 43(3):A2130--A2160, 2021.

\bibitem{5}
M.~Discacciati and A.~Quarteroni.
\newblock Convergence analysis of a subdomain iterative method for the finite
  element approximation of the coupling of stokes and darcy equation.
\newblock {\em Computing and Visualization in Science}, 6:93--103, 2004.

\bibitem{1}
M.~Discacciati and A.~Quarteroni.
\newblock Navier-stokes/darcy coupling: modeling, analysis, and numerical
  approximation.
\newblock {\em Revista Matematica Complutense}, 22(2):315--426, 2009.

\bibitem{14}
G.~Z. Du and L.~Y. Zuo.
\newblock Local and parallel finite element method for the mixed
  navier-stokes/darcy model with beavers-joseph interface conditions.
\newblock {\em Acta Mathematica Scientia}, 37(5):1331--13479, 2017.

\bibitem{32}
L.~Gao and J.~Li.
\newblock A decoupled stabilized finite element method for the
  dual-porosity-navier-stokes fluid flow model arising in shale oil.
\newblock {\em Numerical Methods for Partial Differential Equations},
  37(3):2357--2374, 2021.

\bibitem{10}
V.~Girault and B.~Rivi$\grave{e}$re.
\newblock Dg approximation of coupled navier-stokes and darcy equations by
  beaver-joseph-saffman interface condition.
\newblock {\em SIAM Journal on Numerical Analysis}, 47(3):2052--2089, 2009.

\bibitem{20}
L.~han, H.~B. Zheng, and W.~J. Layton.
\newblock A decoupling method with different subdomain time steps for the
  nonstationary stokes-darcy model.
\newblock {\em Numerical Methods for Partial Differential Equations},
  29(2):549--583, 2013.

\bibitem{35}
A.~Hay, S.~$\acute{E}$tienne, D.~Pelletier, and A.~Garon.
\newblock hp-adaptive time integration based on the bdf for viscous flows.
\newblock {\em Journal of Computational Physics}, 291:151--176, 2015.

\bibitem{12}
X.~M. He, J.~Li, Y.~P. Lin, and J.~Ming.
\newblock A domain decomposition method for the steady-state
  navier-stokes-darcy model with beavers-joseph interface condition.
\newblock {\em SIAM Journal on Scientific Computing}, 37(5):S264--S290, 2015.

\bibitem{41}
F.~Hecht.
\newblock New development in freefem++.
\newblock {\em Journal of numerical mathematics}, 20(3):251--266, 2012.

\bibitem{9}
Y.~R. Hou and Y.~Qin.
\newblock On the solution of coupled stokes/darcy model with beavers-joseph
  interface condition.
\newblock {\em Computers and Mathematics with Applications}, 77(1):50--65,
  2019.

\bibitem{28}
J.~Li.
\newblock {\em Numerical Methods for the Incompressible Navier-Stokes
  Equations}.
\newblock Science Press, Beijing, 2019.

\bibitem{29}
J.~Li, Y.~X. Bai, and X.~Zhao.
\newblock {\em Modern Numerical Methods for Mathematical Physics Equations}.
\newblock Science Press, Beijing, 2022.

\bibitem{27}
J.~Li, X.~L. Lin, and Z.~X. Chen.
\newblock {\em Finite Volume Methods for the Incompressible Navier-Stokes
  Equations}.
\newblock Springer, Berlin, 2022.

\bibitem{23}
Y.~Li, Y.~R. Hou, W.~J. Layton, and H.~Y. Zhao.
\newblock Adaptive partitioned methods for the time-accurate approximation of
  the evolutionary stokes-darcy system.
\newblock {\em Computer Methods in Applied Mechanics and Engineering},
  364:112923, 2020.

\bibitem{33}
Y.~Li, Y.~R. Hou, and Y.~Rong.
\newblock A second-order artificial compression method for the evolutionary
  stokes-darcy system.
\newblock {\em Numerical Algorithms}, 84:1019--1048, 2020.

\bibitem{26}
Z.~Y. Li. and H.~L. Liao.
\newblock Stability of variable-step bdf2 and bdf3 methods.
\newblock {\em SIAM Journal on Numerical Analysis}, 60(4):2253--2272, 2022.

\bibitem{30}
X.~Liu, J.~Li, and Z.~X. Chen.
\newblock A nonconforming virtual element method for the stokes problem on
  general meshes.
\newblock {\em Computer Methods in Applied Mechanics and Engineering},
  320:694--711, 2017.

\bibitem{3}
M.~Mu and X.~H. Zhu.
\newblock Decoupled schemes for a non-stationary mixed stokes-darcy model.
\newblock {\em Mathematics of Computation}, 79(270):707--731, 2010.

\bibitem{6}
Y.~Qin, L.~L. Chen, Y.~Wang, Y.~Li, and J.~Li.
\newblock An adaptive time-stepping dln decoupled algorithm for the coupled
  stokes-darcy model.
\newblock {\em Applied Numerical Mathematics}, 188:106--128, 2023.

\bibitem{15}
Y.~Qin and Y.~R. Hou.
\newblock Optimal error estimates of a decoupled scheme based on two-grid
  finite element for mixed navier-stokes/darcy model.
\newblock {\em Acta Mathematica Scientia}, 38(4):1361--13699, 2018.

\bibitem{17}
Y.~Qin and Y.~R. Hou.
\newblock The time filter for the non-stationary coupled stokes/darcy model.
\newblock {\em Applied Numerical Mathematics}, 146:260--275, 2019.

\bibitem{18}
Y.~Qin and Y.~R. Hou.
\newblock Error estimates of a second-order decoupled scheme for the
  evolutionary stokes-darcy system.
\newblock {\em Applied Numerical Mathematics}, 154:129--148, 2020.

\bibitem{22}
Y.~Qin, Y.~Wang, L.~L. Chen, Y.~Li, and J.~Li.
\newblock A second-order adaptive time filter algorithm with different
  subdomain variable time steps for the evolutionary stokes/darcy model.
\newblock {\em Computers and Mathematics with Applications}, 150:170--195,
  2023.

\bibitem{19}
Y.~Qin, Y.~S. Wang, and J.~Li.
\newblock A variable time step time filter algorithm for the geothermal system.
\newblock {\em SIAM Journal on Numerical Analysis}, 60(5):2781--2806, 2022.

\bibitem{11}
B.~Rivi$\grave{e}$re.
\newblock Analysis of a discontinuous finite element method for the coupled
  stokes and darcy problems.
\newblock {\em Journal of Scientific Computing}, 22:479--500, 2005.

\bibitem{2}
E.~Rohan, J.~Turjanicov$\acute{a}$, and V.~Luke$\check{s}$.
\newblock Multiscale modelling and simulations of tissue perfusion using the
  biot-darcy-brinkman model.
\newblock {\em Computers and Structures}, 251:106404, 2021.

\bibitem{24}
J.~J. Wang.
\newblock Superconvergence analysis of an energy stable scheme with three step
  backward differential formula-finite element method for nonlinear
  reaction-diffusion equation.
\newblock {\em Numerical Methods for Partial Differential Equations},
  39(1):30--44, 2023.

\bibitem{42}
Y.~Z. Zhang, Y.~R. Hou, and J.~P. Zhao.
\newblock Error analysis of a fully discrete finite element variational
  multiscale method for the natural convection problem.
\newblock {\em Computers and Mathematics with Applications}, 68(4):543--567,
  2014.

\bibitem{16}
L.~Y. Zuo and G.~Z. Du.
\newblock A multi-grid technique for coupling fluid flow with porous media
  flow.
\newblock {\em Computers and Mathematics with Applications}, 75(11):4012--4021,
  2018.

\end{thebibliography}

\section{Appendix}
\label{sec6:appendix a}
\textbf{Appendix A}

We have observed that incorporating a time filter after the BDF3 scheme in the Stokes-Darcy model leads to an enhancement of the original third-order algorithm, elevating it to fourth-order accuracy. To validate this improvement, we will continue utilizing the numerical example presented in the preceding section.
~\\
~\\
\textbf{BDF3 plus Time Filter (Constant BDF3-TF)}
\begin{figure}[t]
 \centering
    \subfloat[BDF3-TF cubical cavity]
    {
        \begin{minipage}[hbpt!]{5cm}
            \centering
            \includegraphics[width=5cm,height=5cm]{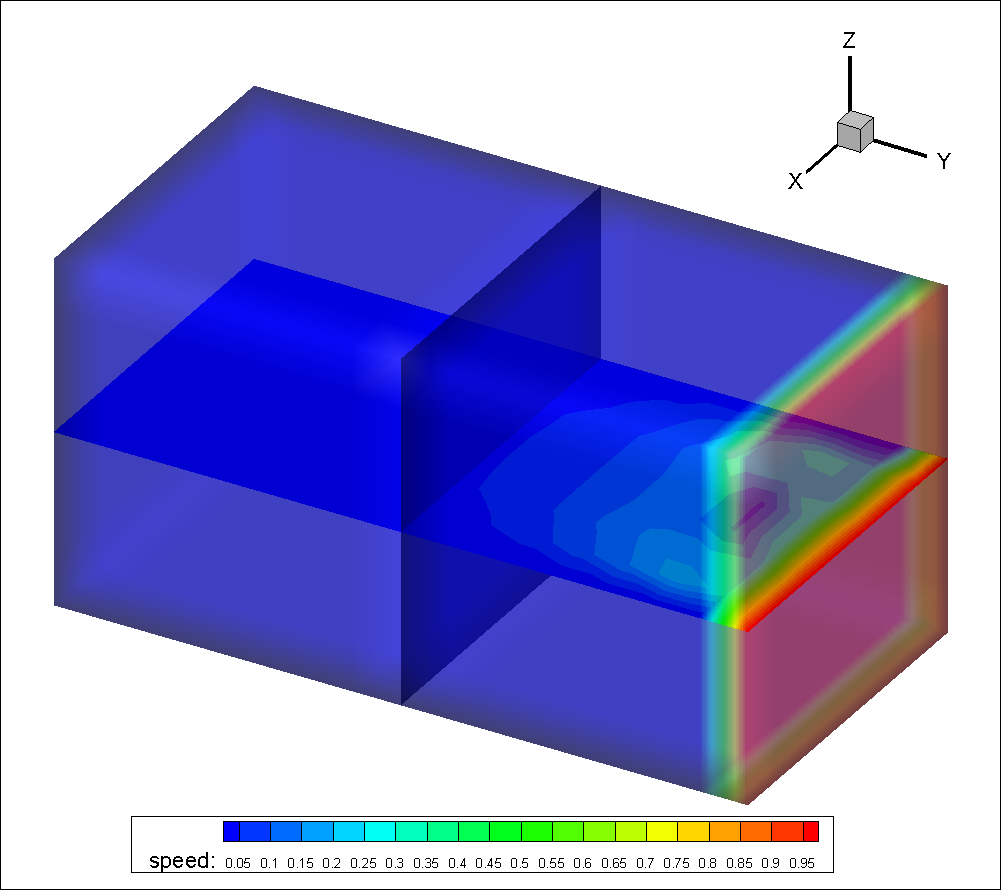}
            \end{minipage}
    }
    \subfloat[BDF3-TF cross-section view]
    {
        \begin{minipage}[hbpt!]{5cm}
            \centering
            \includegraphics[width=5cm,height=5cm]{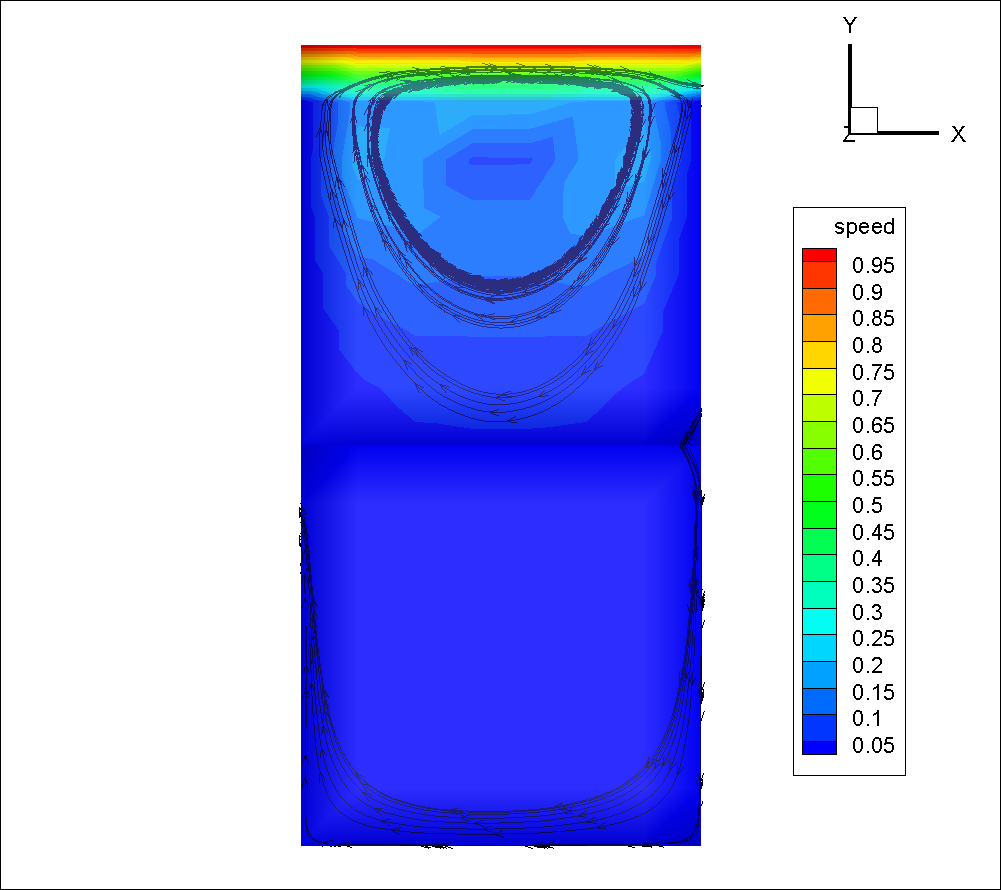}
            \end{minipage}
    }
\caption{\label{fig14} \small The flow line and velocity size of BDF3-TF with variable time-stepping $k_{n-2}$ at $z=0.5$.}
\centering
    \subfloat[BDF3-TF cubical cavity]
    {
        \begin{minipage}[hbpt!]{5cm}
            \centering
            \includegraphics[width=5cm,height=5cm]{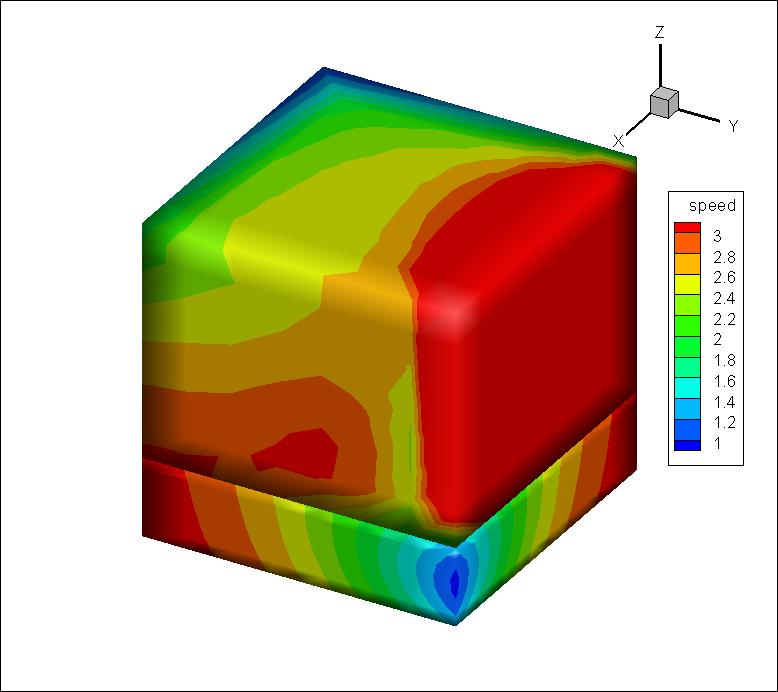}
            \end{minipage}
    }
    \subfloat[BDF3-TF cross-section view]
    {
        \begin{minipage}[hbpt!]{5cm}
            \centering
            \includegraphics[width=5cm,height=5cm]{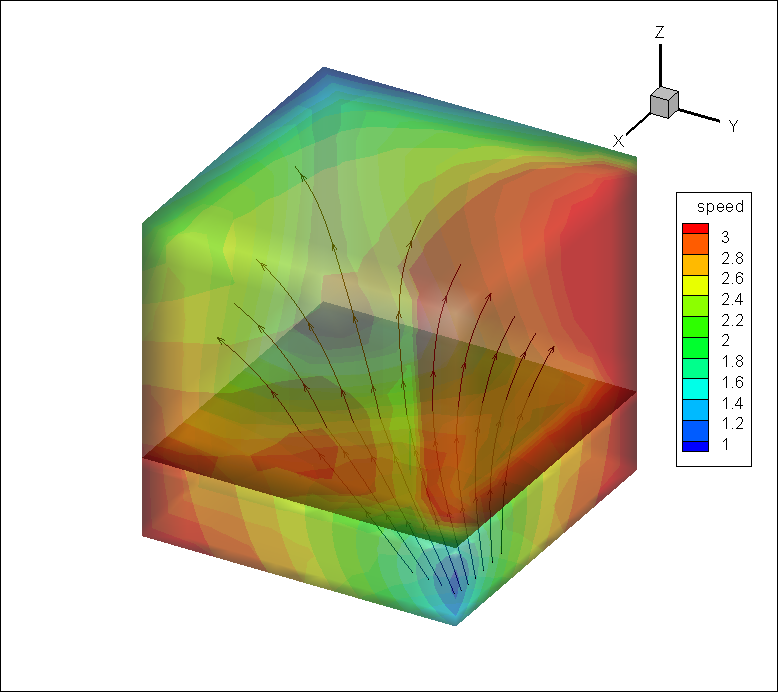}
            \end{minipage}
    }
\caption{\label{fig15} \small The flow line and velocity size of BDF3-TF with variable time-stepping $k_{n-2}$.}
\end{figure}

1. Give $\left(\u_h^0,p_h^0,\phi_h^0\right)$, $\left(\u_h^1,p_h^1,\phi_h^1\right)$, $\left(\u_h^2,p_h^2,\phi_h^2\right)$, and $\left(\u_h^3,p_h^3,\phi_h^3\right)$, find $\left(\hat{\u}_h^{n+1},\hat{p}_h^{n+1},\hat{\phi}_h^{n+1}\right)\in \left(H_{fh},Q_h,H_{ph}\right)$, with $n=3,\ldots,N-1$,
\begin{align}\label{c3}
&\left(\frac{11\hat{\u}_h^{n+1}-18\u_h^{n}+9\u_h^{n-1}-2\u_h^{n-2}}{6k_{n+1}},\vv_h\right) +a_f\left(\hat{\u}_h^{n+1},\vv_h\right)+b\left(\vv_h,\hat{p}_h^{n+1}\right)\nonumber\\
=&\left(F_1^{n+1},\vv_h\right)-c_{\Gamma}\left(\vv_h,4\phi_{h}^{n}-6\phi_{h}^{n-1}+4\phi_h^{n-2}-\phi_h^{n-3}\right),\nonumber\\
&b(\hat{\u}_h^{n+1},q_h)=0;
\end{align}
\begin{align}
&gS\left(\frac{11\hat{\phi}_h^{n+1}-18\phi_h^{n}+9\phi_h^{n_1}-2\phi_h^{n-2}}{6k_{n+1}},\varphi_h\right) +a_p\left(\hat{\phi}_h^{n+1},\varphi_h\right)\nonumber\\
=&g\left(F_2^{n+1},\varphi_h\right)+c_{\Gamma}\left(4\u_{h}^{n}-6\u_{h}^{n-1}+4\u_h^{n-2}-\u_h^{n-3},\varphi_h\right).
\end{align}
\\
2. Apply time filter to update the previous solution
\begin{align}\label{c31}
\u_h^{n+1}&=\hat{\u}_h^{n+1}-\frac{3}{25}\left(\hat{\u}_h^{n+1}-4\u_h^{n}+6\u_h^{n-1}-4\u_h^{n-2}+\u_{h}^{n-3}\right),\nonumber\\
p_h^{n+1}&=\hat{p}_h^{n+1}-\frac{3}{25}\left(\hat{p}_h^{n+1}-4p_h^{n}+6p_h^{n-1}-4p_h^{n-2}+p_{h}^{n-3}\right),\\
\phi_h^{n+1}&=\hat{\phi}_h^{n+1}-\frac{3}{25}\left(\hat{\phi}_h^{n+1}-4\phi_h^{n}+6\phi_h^{n-1}-4\phi_h^{n-2}+\phi_{h}^{n-3}\right).\nonumber
\end{align}

In Figure \ref{fig14}, the velocity flow of BDF3-TF is accurately simulated through velocity flow graphs representing various true solutions.

\begin{table}[hbpt!]
\caption{The convergence performance and CPU time for constant time step BDF3-TF algorithm}
\label{tab:8}
\centering\begin{tabular}{cccccc}
\hline
$\qquad\Delta t\qquad$ & $\quad\rho_{\u}\quad$ &$\quad$ Rate$\quad$  & $\quad\rho_{\phi}\quad$ &$\quad$ Rate $\quad$&$\qquad$ CPU(s)$\qquad$\\
\noalign{\smallskip}\hline\noalign{\smallskip}
$\frac{1}{100}$ &   1.41929$\mathrm{e}-8$ & -  & 3.93736$\mathrm{e}-8$ & - & 1071.68 \\
$\frac{1}{200}$ & 8.73591$\mathrm{e}-9$ & 16.2466 & 2.42401$\mathrm{e}-9$ & 16.2432  & 2225.69 \\
$\frac{1}{400}$ & 5.44143$\mathrm{e}-10$ & 16.0544 &  1.49886$\mathrm{e}-10$ & 16.1723 & 4216.98 \\
$\frac{1}{800}$ & 3.33902$\mathrm{e}-11$ & 16.29659 & 9.32262$\mathrm{e}-11$ & 16.0777 & 8550.58 \\
\hline
\end{tabular}
\end{table}

For the constant time step BDF3-TF algorithms, we choose uniform triangular meshes with a fixed mesh size $h=\frac{1}{100}$. In this context, we calculate the convergence order with respect to the time step $\Delta t$, so we define
$$\rho_v=\frac{\left\|v_h^{\Delta t}-v_h^{\frac{\Delta t}{2}}\right\|_0}{\left\|v_h^{\frac{\Delta t}{2}}-v_h^{\frac{\Delta t}{4}}\right\|_0}$$
when $\rho\approx 16$ for $\kappa=4$, the corresponding order of convergence in time is of $O(\Delta t^4)$. The findings presented in Table \ref{tab:8} demonstrate the attainment of fourth-order convergence, and we intend to conduct further theoretical analysis in due course.
~\\
~\\
\textbf{Appendix B}
\label{sec6:appendix b}

\textbf{Proof of Lemma \ref{ab}.}
\begin{proof} By Taylor’s theorem with the integral remainder,
\begin{align*}
&\mathcal{A}(\u(t_{n+1}))-k_{n+1}\u_t(t_{n+1}) =a_3\u(t_{n+1})+a_2\u(t_n)+a_1\u(t_{n-1})+a_0\u(t_{n-2})-k_{n+1}\u_t(t_{n+1})\\
=&a_3\u(t_{n+1})-k_{n+1}\u_t(t_{n+1})+a_2\left(\u(t_{n+1})-k_{n+1}\u_t(t_{n+1}) +\frac{(k_{n+1}+k_{n})^2}{2}\u_{tt}(t_{n+1})\right.\\
&\left.-\frac{(k_{n+1}+k_n+k_{n-1})^2}{3}\u_{ttt}(t_{n+1}) +\frac{1}{2}\int_{t_{n+1}}^{t_n}\u_{tttt}(t)(t_n-t)^3dt\right)\\
&+ a_1\left(\u(t_{n+1})-2(k_{n+1}+k_n)\u_t(t_{n+1})+2(k_{n+1}+k_n)^2\u_{tt}(t_{n+1})\right.\\
&\left.-\frac{8(k_{n+1}+k_n)^3}{3}\u_{ttt}(t_{n+1}) +\frac{1}{2}\int_{t_{n+1}}^{t_n-1}\u_{tttt}(t)(t_{n-1}-t)^3dt\right)\\
&+a_0\left(\u(t_{n+1})-3(k_{n+1}+k_n+k_{n-1})\u_t(t_{n+1})+\frac{9(k_{n+1}+k_n+k_{n-1})^2}{2} \u_{tt}(t_{n+1})\right.\\
&\left.-9(k_{n+1}+k_n+k_{n-1})^3\u_{ttt}(t_{n+1}) +\frac{1}{2}\int_{t_{n+1}}^{t_{n-2}}\u_{tttt}(t)(t_{n-2}-t)^3dt\right)\\
\leq &C\left(\int^{t_{n+1}}_{t_n}\u_{tttt}(t)(t_n-t)^3dt +\int^{t_{n+1}}_{t_n}\u_{tttt}(t)(t_n-t)^3dt +\int^{t_{n+1}}_{t_{n-2}}\u_{tttt}(t)(t_{n-2}-t)^3dt\right).
\end{align*}
These terms are first estimated by Cauchy-Schwarz,
\begin{align*}
&\left(\int^{t_{n+1}}_{t_n}\u_{tttt}(t)(t_n-t)^3dt\right)^2 \leq\int^{t_{n+1}}_{t_n}\u_{tttt}^2dt\int^{t_{n+1}}_{t_n}(t_n-t)^6dt \leq C k_{n+1}^7\int^{t_{n+1}}_{t_n}\u_{tttt}^2dt,\\
&\left(\int^{t_{n+1}}_{t_{n-1}}\u_{tttt}(t)(t_{n-1}-t)^3dt\right)^2 \leq\int^{t_{n+1}}_{t_{n-1}}\u_{tttt}^2dt\int^{t_{n+1}}_{t_{n-1}}(t_{n-1}-t)^6dt \leq Ck_{n+1}^7\int^{t_{n+1}}_{t_{n-1}}\u_{tttt}^2dt,\\
&\left(\int^{t_{n+1}}_{t_{n-2}}\u_{tttt}(t)(t_{n-2}-t)^3dt\right)^2 \leq\int^{t_{n+1}}_{t_{n-2}}\u_{tttt}^2dt\int^{t_{n+1}}_{t_{n-2}}(t_{n-2}-t)^6dt \leq Ck_{n+1}^7\int^{t_{n+1}}_{t_{n-2}}\u_{tttt}^2dt.
\end{align*}
Thus,
\begin{align*}
&\left(\frac{\mathcal{A}(\u(t_{n+1}))}{k_{n+1}}-\u_t(t_{n+1})\right)^2\leq Ck_{n+1}^5\int^{t_{n+1}}_{t_{n-2}}\u_{tttt}^2dt.
\end{align*}

Using integrating by parts and the Cauchy-Schwarz inequality, we obtain
\begin{align*}
&\left\|\frac{\mathcal{A}(\tilde{\u}^{n+1})}{k_{n+1}}-\frac{\mathcal{A}(\u(t_{n+1}))}{k_{n+1}}\right\|_f^2 =\left\|(P_h^{\u}-I)\frac{a_3\u(t_{n+1})+a_2\u(t_n)+a_1\u(t_{n-1})+a_0\u(t_{n-2})}{k_{n+1}}-\right\|_f^2\\
\leq& \frac{C}{k_{n+1}^2}\int_{\Omega_f}\left(\int_{t_n}^{t_{n+1}}\beta_3(P_h^{\u}-I)\u_tdt -\int_{t_{n-1}}^{t_{n}}\beta_2(P_h^{\u}-I)\u_tdt +\int_{t_{n-2}}^{t_{n-1}}\beta_1(P_h^{\u}-I)\u_tdt\right)^2dx\\
\leq&\frac{C}{k_{n+1}^2}\int_{\Omega_f}\left(k_{n+1}\int_{t_n}^{t_{n+1}}\beta_3^2\left\|(P_h^{\u}-I)\u_t\right\|_f^2dt +k_{n}\int_{t_{n-1}}^{t_{n}}\beta_2^2\left\|(P_h^{\u}-I)\u_t\right\|_f^2dt\right.\\ &\left.+k_{n-1}\int_{t_{n-2}}^{t_{n-1}}\beta_1^2\left\|(P_h^{\u}-I)\u_t\right\|_f^2dt\right)dx\\
\leq&\frac{C}{k_{n+1}}\left(\int_{t_n}^{t_{n+1}}\left\|(P_h^{\u}-I)\u_t\right\|_f^2dt +\int_{t_{n-1}}^{t_{n}}\left\|(P_h^{\u}-I)\u_t\right\|_f^2dt +\int_{t_{n-2}}^{t_{n-1}}\left\|(P_h^{\u}-I)\u_t\right\|_f^2dt\right)\\
\leq&\frac{C}{k_{n+1}}\int_{t_{n-2}}^{t_{n+1}}\left\|(P_h^{\u}-I)\u_t\right\|_f^2dt.
\end{align*}

The last inequality is proved similarly
\begin{align*}
&\left\|\mathcal{B}(\tilde{\u}^{n+1})-\tilde{\u}_{\sigma}^{n}\right\|_{H_f}^2
\leq C\left\|\mathcal{B}(\u^{n+1})-\u_{\sigma}^{n}\right\|_{H_f}^2
\leq C\int_{\Omega_f}\left|C\int_{t_{n-1}}^{t_{n+1}}\nabla\u_{tt}dt +C\int_{t_{n-2}}^{t_{n}}\nabla\u_{tt}dt\right|^2dx\\
\leq&C\int_{\Omega_f}\left|Ck_{n+1}\int_{t_{n-2}}^{t_{n+1}}\nabla\u_{ttt}dt +C\int_{t_{n-2}}^{t_{n}}(t_n-t)^2\nabla\u_{ttt}dt +C\int_{t_{n-1}}^{t_{n+1}}(t_{n-1}-t)^2\nabla\u_{ttt}dt\right|^2dx\\
\leq&Ck_{n+1}^5\int_{t_{n-2}}^{t_{n+1}}\left\|\u_{ttt}\right\|_{H_f}^2.
\end{align*}
Thus, we can derive the inequalities stated in Lemma \ref{ab}.
\end{proof}

\end{document}